\documentclass[12pt]{article}
\usepackage{amsmath}
\usepackage{amssymb}
\usepackage{amsthm}
\usepackage{yhmath}
\usepackage{mathrsfs}
\usepackage{graphicx}
\usepackage{epstopdf}
\usepackage{multirow}
\usepackage{longtable}
\usepackage{mathtools}
\usepackage{subfigure}
\usepackage{float}
\usepackage{geometry}
\usepackage{galois}
\usepackage{color}
\usepackage{diagbox}
\usepackage{empheq}
\usepackage{textcomp}
\usepackage[greek,english]{babel}
\usepackage{framed}
\usepackage{bm}
\usepackage{appendix}
\usepackage{cite}
\usepackage[numbers,compress]{natbib}

\makeatletter
\@addtoreset{equation}{section}
\makeatother
\definecolor{shadecolor}{gray}{0.85}
\newcommand{\arraystrech}[5]

\newtheorem{Lemma}{Lemma}
\newtheorem{Theorem}{Therorem}

\newtheorem{rem}{Remark}

\DeclareSymbolFont{largesymbol}{OMX}{yhex}{m}{n}
\DeclareMathAccent{\widehat}{\mathord}{largesymbol}{"62}

\begin{document}
\pagestyle{plain}

\title{A Mixed Finite Element Method for Multi-Cavity Computation in Incompressible
Nonlinear Elasticity\thanks{The research was supported by the NSFC
projects 11171008 and 11571022.}}

\author{Weijie Huang, \hspace{1mm} Zhiping Li\thanks{Corresponding author,
email: lizp@math.pku.edu.cn} \\ {\small LMAM \& School of Mathematical
Sciences, Peking University, Beijing 100871, China}}

\date{}

\maketitle
\begin{abstract}
A mixed finite element method combining an
iso-parametric $Q_2$-$P_1$ element and an iso-parametric $P_2^+$-$P_1$
element is developed for the computation of multiple cavities in incompressible
nonlinear elasticity. The method is analytically proved to be locking-free and convergent,
and it is also shown to be numerically accurate and efficient by numerical experiments.
Furthermore, the newly developed accurate method enables us to find an interesting new bifurcation
phenomenon in multi-cavity growth.
\end{abstract}


\section{Introduction}

Cavitation phenomenon, which exhibits sudden dramatic growth of pre-exist small voids under loads
exceeding certain criteria, is first systematically modeled and analyzed by Gent \& Lindley
\cite{Gent and Lindley} in 1958. It is considered one of the most important failure phenomenon
in nonlinear elasticity, and its better understanding is crucial to explore the properties
of elastic materials.

Let $\Omega \subset \mathbb{R}^{n}\,(n=2,3)$ be a simply connected domain with sufficiently
smooth boundary $\partial \Omega$, and let
$B_{\rho_k}(\bm{x}_k)=\{\bm{x}\in\mathbb{R}^n:|\bm{x}-\bm{x}_k|<\rho_k\}$ and
$\bigcup_{k=1}^{K} B_{\rho_k}(\bm{x}_k)\subset \Omega$. Let
$\Omega_{\rho} = \Omega \setminus \bigcup_{k=1}^{K} B_{\rho_k}(\bm{x}_k)$
be the domain occupied by an elastic body in its reference
configuration, where $B_{\rho_k}(\bm{x}_k)$ denotes the pre-existing defects of radii
$\rho_k\ll 1$ centered at $\bm{x}_k$, $k=1,\cdots,K$. Then, in incompressible elastic
materials, the multi-cavitation problem can be expressed
as to find a deformation $\bm{u}$ to minimize the total energy
\begin{equation}
  \label{functional}
  E(\bm{u})=\int_{\Omega_{\rho}}W_0(\nabla \bm{u}(\bm{x})) \,
  \mathrm{d}\bm{x},
\end{equation}
in the set of admissible deformation functions
\begin{equation}
  \label{admissible set incompressible}
  \mathcal{A}_I=\{\bm{u}\in W^{1,s}(\Omega_{\rho};\mathbb{R}^n) \ \mbox{is 1-to-1
  a.e.}: \bm{u}|_{\partial\Omega}=\bm{u}_0,\; \det \nabla \bm{u} =1, \mbox{ a.e.} \},
\end{equation}
where $W_0: M^{n \times n}_+\rightarrow\mathbb{R}^+$ is the stored energy
density function of the material with $M_+^{n \times n}$ being the set of $n \times n$
matrices of positive determinant, and $n-1<s<n$ is a given Sobolev index,
and where a displacement boundary condition $\bm{u}=\bm{u}_0$ is imposed on
$\partial_D\Omega_{\rho} = \partial \Omega$, and a traction free boundary
condition is imposed on $\partial_N\Omega_{\rho} = \cup_{k=1}^K B_{\rho_k}(\bm{x}_k)$.
Without loss of generality, we consider a typical energy density
for nonlinear elasticity given as
\begin{equation}
\label{energy density}
\begin{aligned}
W(F) = \mu |F|^s + d(\det F),\quad \forall F\in M_+^{n\times n},
\end{aligned}
\end{equation}
where $\mu$ is material parameter, and $d(\det F) = \kappa (\det F - 1)^2 + d_1(\det F)$
with $\kappa>0$ and $d_1:\mathbb{R}_+\to\mathbb{R}_+$ being a strictly convex function satisfying
\begin{equation}
\begin{aligned}
\label{property of g}
d_1(\xi)\to+\infty\ \text{as}\ \xi\to 0,\ \text{and}\
\dfrac{d_1(\xi)}{\xi}\to+\infty \ \text{as}\ \xi\to +\infty.
\end{aligned}
\end{equation}
Notice that, even though for incompressible nonlinear elastic materials, $d(\cdot)$ is
just a constant as the determinant of any admissible deformation in $\mathcal{A}_I$
equals 1 a.e., the term plays an important role in the proof of the convergence of
the numerical cavitation solutions to the mixed formulation given below.

To relax the rather restrictive condition $\det \nabla \bm{u} =1, \mbox{ a.e.}$ appeared in
$\mathcal{A}_I$, a mixed formulation of the following
form (see \cite{Fortion1991,Zienkiewicz}) is usually used in computation:
\begin{equation}
\label{mixed formulation}
(\bm{u},p) = \arg \sup_{p \in L^2(\Omega_{\rho})} \inf_{\bm{u} \in \mathcal{A}}
E(\bm{u},p),
\end{equation}
where $p\in L^2(\Omega_{\rho})$ is a pressure like Lagrangian multiplier introduced
to relax the constraint of incompressibility, and where the Lagrangian functional $E(\bm{u},p)$
and the set of admissible deformations $\mathcal{A}$ are defined as
\begin{equation}
\label{mixed functional}
\begin{aligned}
E(\bm{u},p)=\int_{\Omega_{\rho}}\left(W(\nabla \bm{u})-p\left(\det\nabla\bm{u}-1
\right)\right) \,\mathrm{d}\bm{x},
\end{aligned}
\end{equation}
\begin{equation}
  \label{admissible set}
\mathcal{A}=\{\bm{u}\in W^{1,s}(\Omega_{\rho};\mathbb{R}^n) \ \mbox{is 1-to-1
a.e.},\bm{u}|_{\partial_D\Omega_\rho} = \bm{u}_0 \}.
\end{equation}

The nonlinear saddle point problem \eqref{mixed functional}-\eqref{admissible set}
with energy density \eqref{energy density} leads to the  mixed displacement/traction
boundary value problem of the Euler-Lagrange equation:
\begin{eqnarray}
\label{Euler Lagrange equation-1}
\operatorname{div} \big(D_F W(\nabla\bm{u}) - p\operatorname{cof}\nabla \bm{u}\big)
&=& 0,\ \quad \text{in}\,\,\,  \Omega_{\rho},\\
\label{Euler Lagrange equation-2}
\det\nabla\bm{u} &=& 1,\ \quad  \text{in}\,\,\, \Omega_{\rho},\\
\label{Euler Lagrange equation-3}
\big(D_F W(\nabla\bm{u})-p\operatorname{cof}\nabla \bm{u}\big)\bm{n} &=& \bm{0},
\ \quad \text{on} \ \cup_{k=1}^K\partial B_{\rho_k}(\bm{x}_k),\\
\label{Euler Lagrange equation-4}
\bm{u} &=& \bm{u}_0, \quad\! \text{on} \ \partial_D\Omega_\rho.
\end{eqnarray}

One of the main difficulties of numerical cavitation computation comes from the very large
anisotropic deformation near the cavities, which, if not properly approximated,
can cause mesh entanglement corresponding to nonphysical material interpenetration.
In recent years, successful quadratic iso-parametric and dual-parametric finite element methods
have been developed for the cavitation computation for compressible nonlinear elastic materials
(\cite{LianLi2011,LianLi20112,SuLiRectan,SuLi2018}). However, direct application of these
methods to the case of incompressible elasticity generally encounters the barrier of
the locking effect. More recently, a dual-parametric mixed finite elements (DP-Q2-P1)
based on the saddle point problem \eqref{mixed functional}-\eqref{admissible set} is
established by Huang and Li \cite{HuangLi2017}, which is shown to be locking-free,
convergent and effective.

In the present paper, we develop a mixed finite element method combining an
iso-parametric $Q_2$-$P_1$ element and an iso-parametric $P_2^+$-$P_1$
element for the computation of multiple cavities in incompressible
nonlinear elasticity. By extending and elaborating the techniques used in \cite{HuangLi2017},
we are able to analytically prove that the method is locking-free and convergent.
A damped Newton method is applied to solve the discrete Euler-Lagrange equation.
Our numerical experiments show that the method is numerically efficient.
Furthermore, the newly developed accurate method enable us to find an interesting new bifurcation
phenomenon in multi-cavity growth. It is worth mentioning here, if the displacement
boundary condition \eqref{Euler Lagrange equation-4} is replaced by a traction boundary
condition, the results of this paper still hold, and the proof is essentially the same.

The rest of the paper is organized as follows. In \S~2, we introduce the iso-parametric
mixed finite element method. The locking-free and stability analysis of the method is
given in \S~3. The convergence analysis of the finite element solutions is
given in \S~4. Numerical experiments
and results are presented in \S~5 to show the accuracy and efficiency of the method.
Some concluding remarks are given in \S~6.

\section{The mixed finite element method}

In this section, we present an iso-parametric $P_2^+$-$P_1$ curve edged triangular
element and an iso-parametric $Q_2$-$P_1$ curve edged rectangular element, and establish
a mixed finite element method by introducing a specially designed mesh to
couple the two elements for the computation of multiple cavities in incompressible
nonlinear elasticity based on a weak form of the Euler-Lagrange equation
\eqref{Euler Lagrange equation-1}-\eqref{Euler Lagrange equation-4}.

\subsection{The iso-parametric $P_2^+$-$P_1$ element}

The standard $P_2^+$-$P_1$ mixed triangular element $(\hat{T}_t, \hat{P}, \hat{\Sigma})$
is defined as
\begin{align*}
\begin{cases}
\hat{T}_t\ \text{is the reference triangular element (see Fig~\ref{reference triangle})},\\
\hat{P}=\{P_2^{+}(\hat{T}_t);\; P_1(\hat{T}_t)\},\\
\hat{\Sigma}=\{\bm{\hat{u}}(\hat{\bm{a}}_i), 0\le i\le 3,\ \text{and}\ 
\bm{\hat{u}}(\hat{\bm{a}}_{ij}), 1\le i< j\le 3;\;
\hat{p}(\hat{\bm{\bm{b}}}_i), 1\le i \le 3\},
\end{cases}
\end{align*}
where $\{\hat{\bm{a}}_i\}_{i=1}^{3}$ are the vertices of $\hat{T}_t$, 
$\hat{\bm{a}}_0=\sum_{i=1}^3 \hat{\bm{a}}_i/3$ is the barycenter of $\hat{T}_t$, 
$\{\hat{\bm{a}}_{ij}\}_{1\le i<j\le3}$
represent the midpoints between $\hat{\bm{a}}_i$ and $\hat{\bm{a}}_j$, and $\hat{\bm{b}}_i$
are three non-collinear interior points (say Gaussian quadrature nodes of $\hat{T}_t$), and
$P_2^{+}(\hat{T}_t)=\text{span}\{ P_2(\hat{T}_t), \hat{\lambda}_1(\hat{\bm{x}})
\hat{\lambda}_2(\hat{\bm{x}})\hat{\lambda}_3(\hat{\bm{x}})\}$ with
$\hat{\lambda}_i(\hat{\bm{x}}), 1\le i\le 3$ being the barycentric coordinates of $\hat{T}_t$.

\begin{figure}[H]
 \begin{minipage}[l]{0.49\textwidth}\hspace*{1mm}
	\begin{minipage}[l]{0.225\textwidth}
    \centering    \includegraphics[width=1.4in]{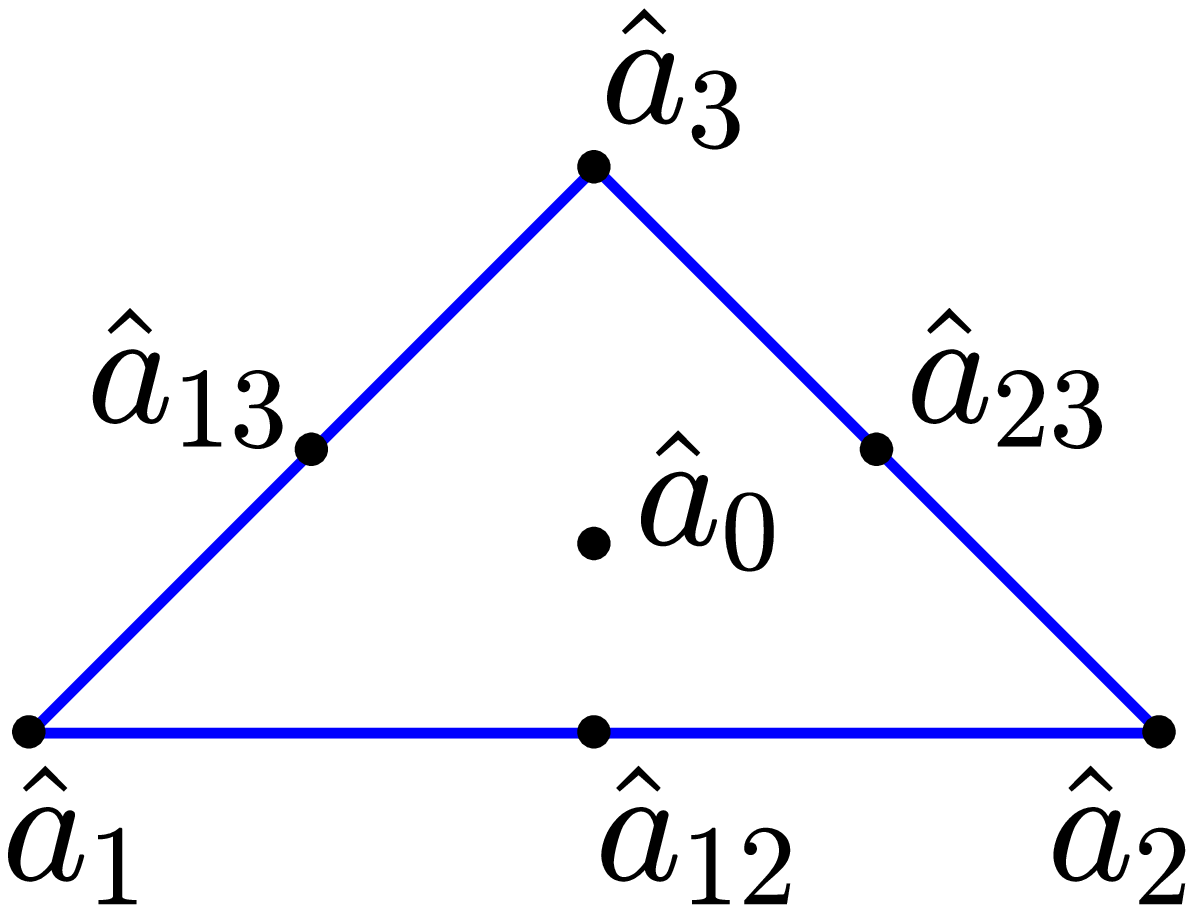}
	\vspace*{-7mm}
    \end{minipage}\hspace*{17mm}
    \begin{minipage}[r]{0.225\textwidth}
    \centering   \includegraphics[width=1.5in]{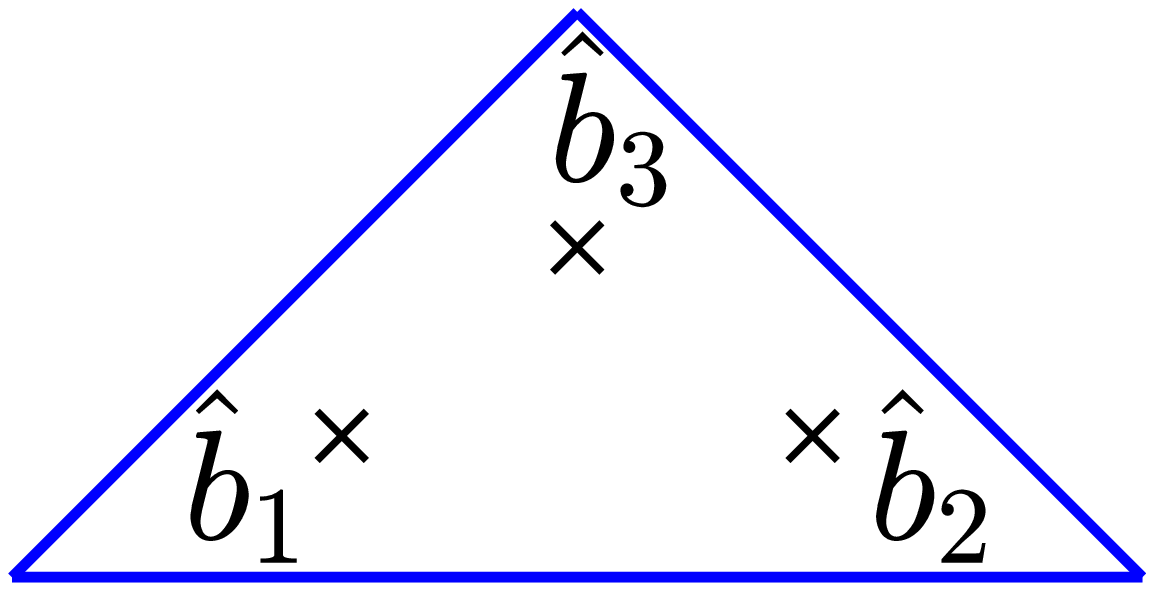}
    \vspace*{-7mm}
    \end{minipage}
    \caption{Reference element $\hat{T}_t$, $\hat{\Sigma}$.}\label{reference triangle}
 \end{minipage}
 \begin{minipage}[r]{0.49\textwidth}
    \begin{minipage}[l]{0.225\textwidth}\vspace*{-3mm}
    \centering    \includegraphics[width=1.65in]{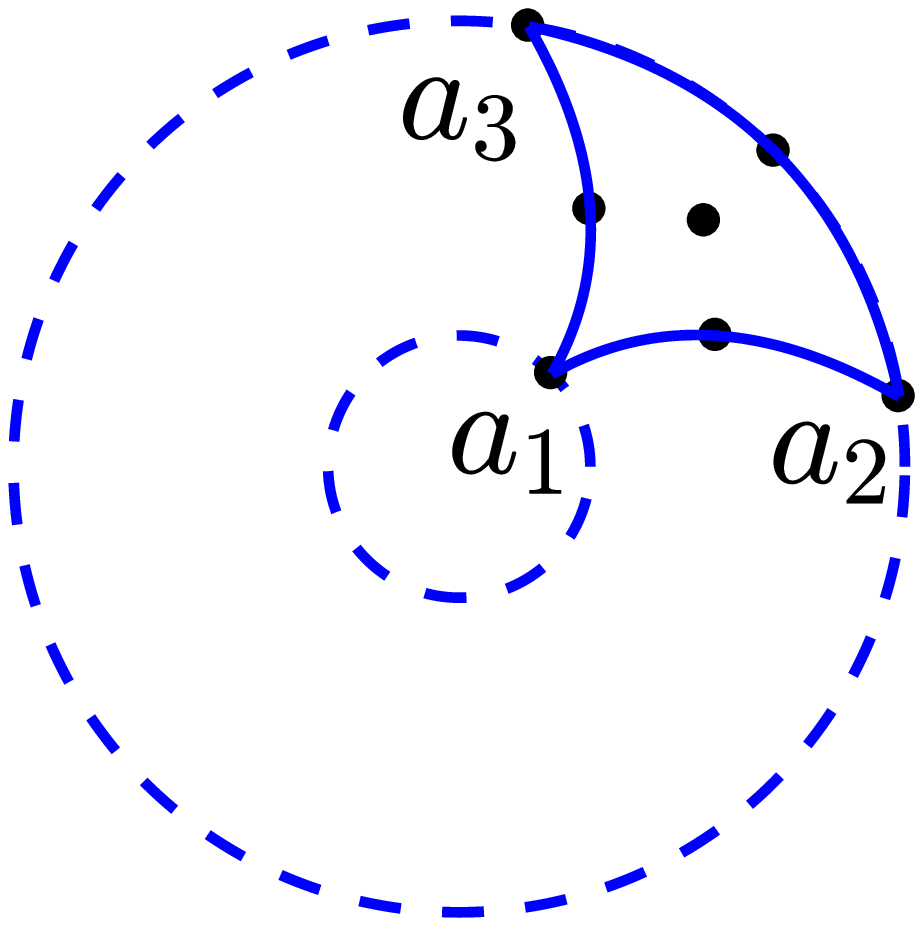}
	\vspace*{-10mm}
    \end{minipage}\hspace*{18mm}\vspace*{3mm}
    \begin{minipage}[r]{0.225\textwidth}\vspace*{-3mm}
    \centering   \includegraphics[width=1.65in]{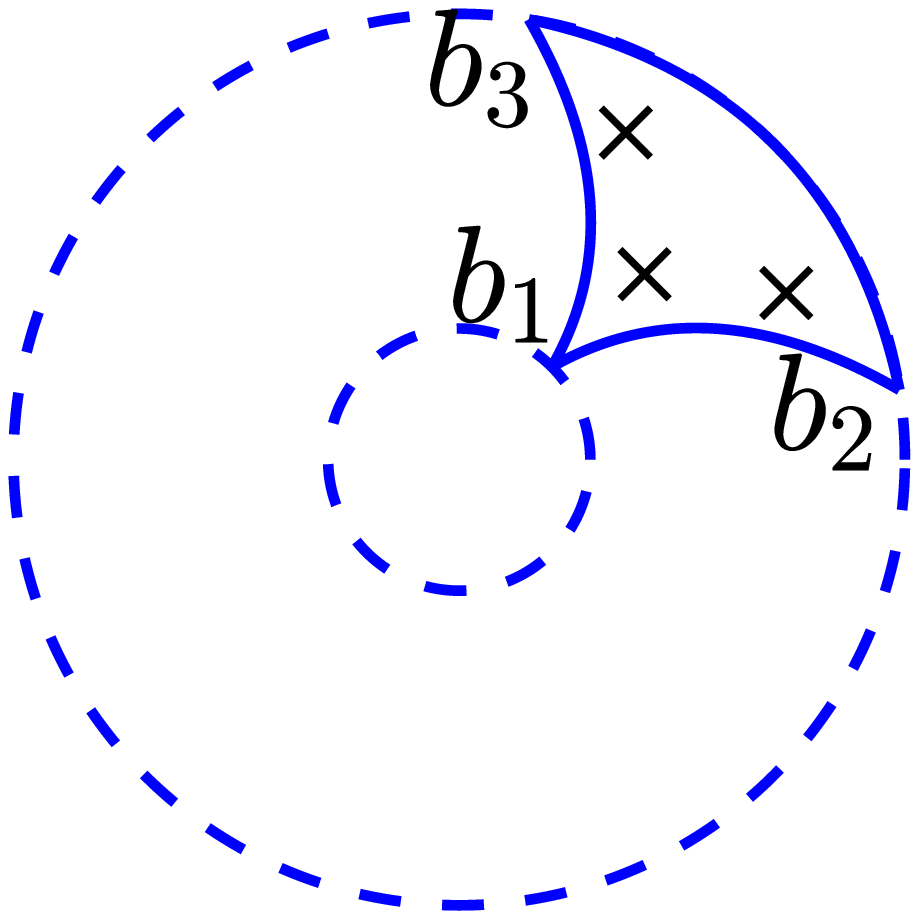}
    \vspace*{-10mm}
    \end{minipage}
    \caption{Element $T_t$, $\Sigma$.}\label{element triangle}
 \end{minipage}
\end{figure}

Given 3 non-collinear anti-clock-wisely ordered vertices $\{\bm{a}_i\}_{i=1}^3$, denote
$(r_k(\bm{x}),\theta_k(\bm{x}))$ the
local polar coordinates of $\bm{x}$ with respect to the nearest defect center $\bm{x}_k$, set
\begin{equation}
\label{midpoints1}
\begin{aligned}
\bm{a}_{ij}=\bm{x}_k+(r_{ij}\cos\theta_{ij},r_{ij}\sin\theta_{ij}),
\end{aligned}
\end{equation}
where
\begin{equation}
\label{midpoints2}
\begin{aligned}
r_{ij} = \dfrac{r(\bm{a}_i)+r(\bm{a}_j)}{2},\quad
\theta_{ij} = \dfrac{\theta(\bm{a}_i)+\theta(\bm{a}_j)}{2}.
\end{aligned}
\end{equation}
Let $F_{T_t}: \hat{T}_t \rightarrow \mathbb{R}^2$ be defined as
\begin{align}
\label{triangular element}
\begin{cases}
F_{T_t}\in (P_2(\hat{T}_t))^2,\\
\bm{x} = F_{T_t}(\hat{\bm{x}})=\sum\limits_{i=0}^{3}\bm{a}_i
\hat{\mu}_i(\hat{\bm{x}})+\sum\limits_{1\le i< j\le 3}\bm{a}_{ij}\hat{\mu}_{ij}(\hat{\bm{x}})
\end{cases}
\end{align}
where $\hat{\mu}_0 = 27\prod_{i=1}^3\hat{\lambda}_i(\hat{\bm{x}})$ and 
$$
\hat{\mu}_i = \hat{\lambda}_i(\hat{\bm{x}})(2\hat{\lambda}_i(\hat{\bm{x}})-1),\; 1\le i\le 3; \;\;
\hat{\mu}_{ij}=4\hat{\lambda}_i(\hat{\bm{x}})\hat{\lambda}_j(\hat{\bm{x}}),
\;\; 1\le i < j \le 3.
$$
If the map $F_{T_t}: \hat{T}_t \rightarrow F_{T_t}(\hat{T}_t)$ is an injection, then
$T_t=F_{T_t}(\hat{T}_t)$ defines a curved triangular element.
We define the
iso-parametric $P_2^+$-$P_1$ mixed finite element $(T_t, P, \Sigma)$ as follows:
\begin{equation*}
\begin{cases}
T_t=F_{T_t}(\hat{T}_t)\ \text{being a curved triangular element (see Fig~\ref{element triangle})},\\
P=\big\{(\bm{u}, p) : T_t\to \mathbb{R}^2\times \mathbb{R} | \bm{u}=\hat{\bm{u}}\comp F_{T_t}^{-1},
\hat{\bm{u}}\in P_2^+; \ p=\hat{p}\comp F_{T_t}^{-1}, \hat{p}\in P_1\big\},\\
\Sigma=\big\{\bm{u}(\bm{a}_i), 0\le i\le 3,\ \text{and}\ \bm{u}(\bm{a}_{ij}),1\le i< j\le 3;\
p(\bm{b}_i),1\le i\le 3\big\}.
\end{cases}
\end{equation*}

\begin{rem}
If $\bm{a}_{ij}$ is defined as $\bm{a}_{ij} = \dfrac{\bm{a}_i+\bm{a}_j}{2}, 1\le i< j\le3$,
then, $T_t$ reduces to a straight edged triangle.
\end{rem}

\subsection{The iso-parametric $Q_2$-$P_1$ element}

Let $(\hat{T}_q, \hat{P}, \hat{\Sigma})$ be the standard biquadratic-linear mixed
rectangular element:
\begin{align*}
\begin{cases}
\hat{T}_q=[-1, 1]\times [-1, 1]\ \text{is the standard reference rectangular
element (see Fig~\ref{reference qual})},\\
\hat{P}=\{Q_2(\hat{T}_q), P_1(\hat{T}_q)\},\\
\hat{\Sigma}=\{\bm{\hat{u}}(\hat{\bm{a}}_i), 0\le i\le 8; \;\;
\hat{p}(\hat{\bm{b}}_0), \partial_{\hat{x}_1} \hat{p}(\hat{\bm{b}}_0),
\partial_{\hat{x}_2} \hat{p}(\hat{\bm{b}}_0) \},
\end{cases}
\end{align*}
where $\{\hat{\bm{a}}_i\}_{i=1}^4$ are the vertices of $\hat{T}_q$, $\{\hat{\bm{a}}_{i}\}_{i=5}^8$
represent the midpoints of the edges of $\hat{T}_q$, $\hat{\bm{a}}_0= \hat{\bm{b}}_0=(0,0)$, and `-' denotes the 1st-order derivatives at $\hat{\bm{b}}_0$ and $\bm{b}_0$.

\begin{figure}[H]
 \begin{minipage}[l]{0.49\textwidth}\hspace*{1mm}
	\begin{minipage}[l]{0.225\textwidth}
    \centering    \includegraphics[width=1.5in]{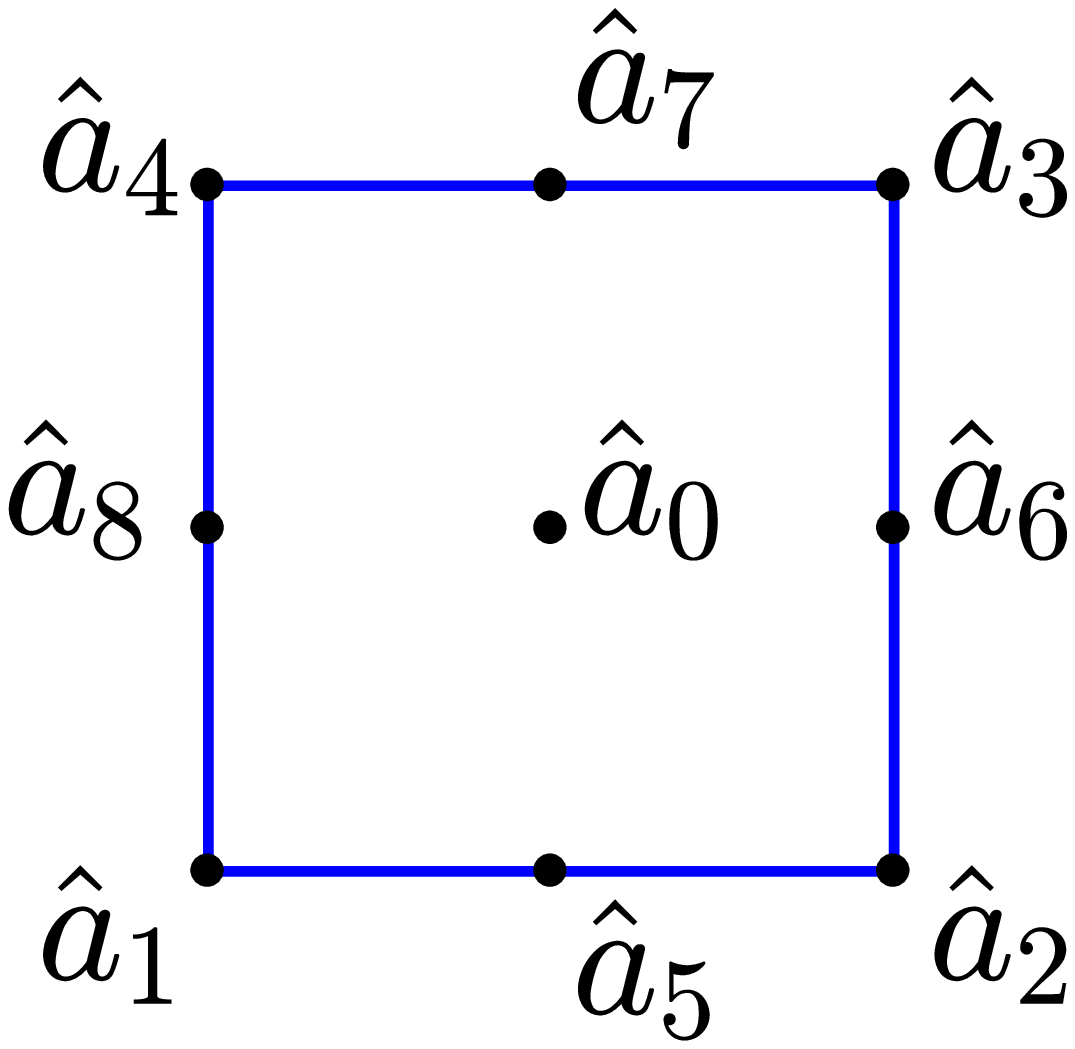}
	\vspace*{-7mm}
    \end{minipage}\hspace*{17mm}
    \begin{minipage}[r]{0.225\textwidth}
    \centering   \includegraphics[width=1.5in]{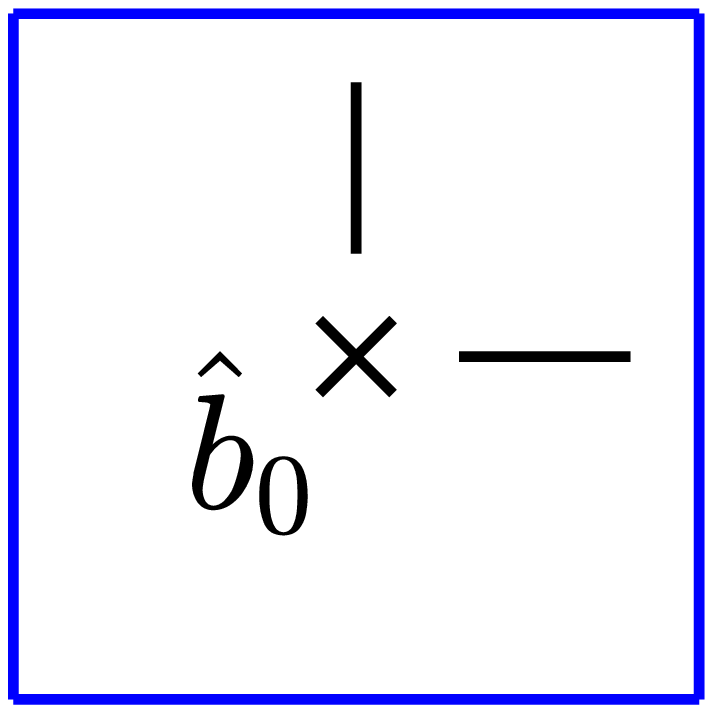}
    \vspace*{-7mm}
    \end{minipage}
    \caption{Reference element $\hat{T}_q$, $\hat{\Sigma}$.}\label{reference qual}
 \end{minipage}
 \begin{minipage}[r]{0.49\textwidth}
    \begin{minipage}[l]{0.225\textwidth}\vspace*{-3mm}
    \centering    \includegraphics[width=1.65in]{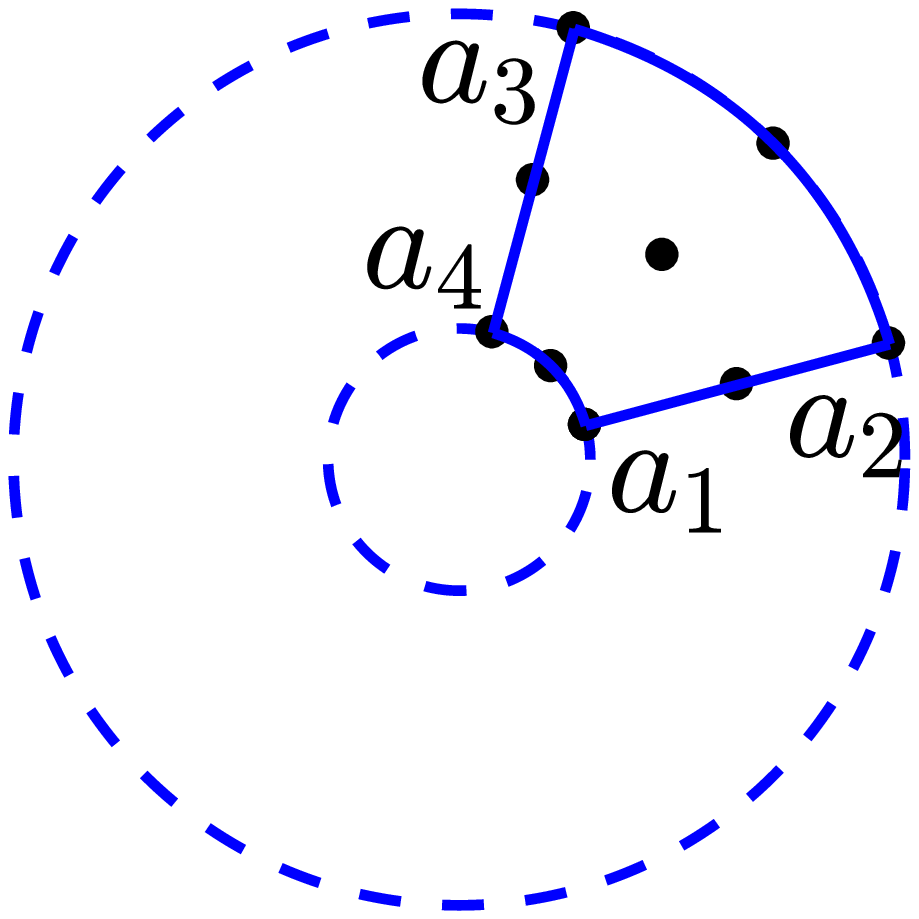}
	\vspace*{-10mm}
    \end{minipage}\hspace*{18mm}\vspace*{3mm}
    \begin{minipage}[r]{0.225\textwidth}\vspace*{-3mm}
    \centering   \includegraphics[width=1.65in]{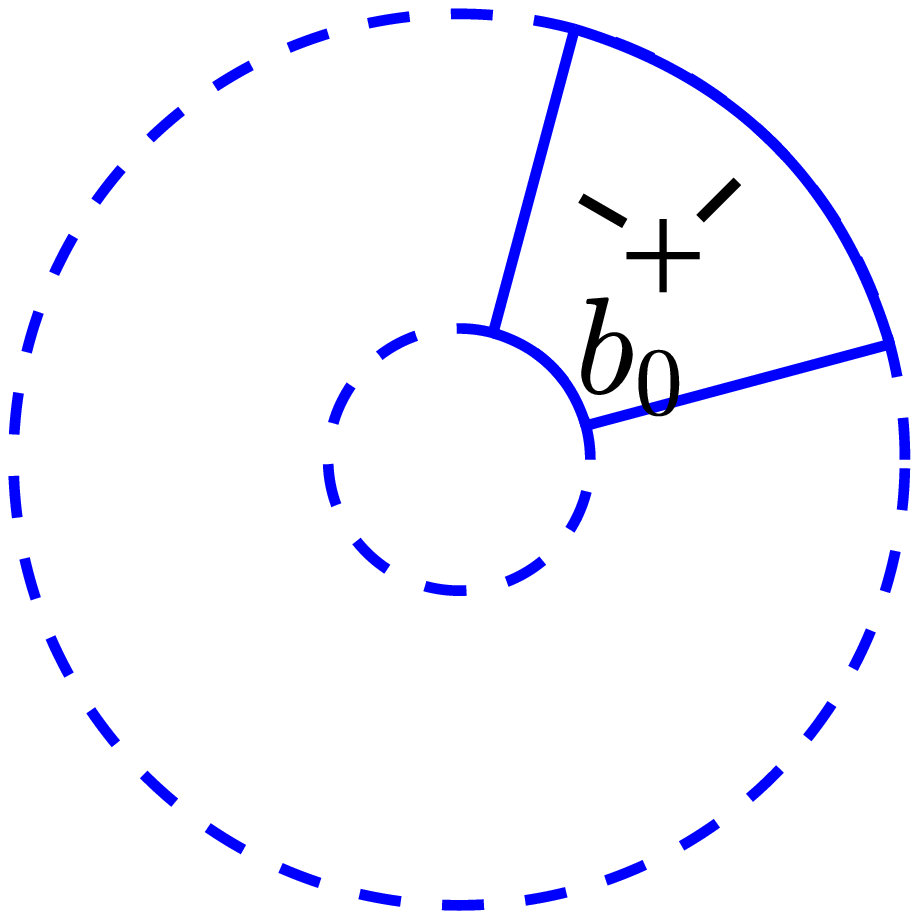}
    \vspace*{-10mm}
    \end{minipage}
    \caption{Element $T_q$, $\Sigma$.}\label{element qual}
 \end{minipage}
\end{figure}

Given 4 non-degenerate vertices $\{\bm{a}_i\}_{i=1}^{4}$, and set
$\{\bm{a}_{i}\}_{i =5}^8$ and $\bm{a}_0$ the same as for the curved triangular
element (see \eqref{midpoints1}-\eqref{midpoints2}).
Define $F_{T_q}:\hat{T}_q\to \mathbb{R}^2$ by
\begin{align}
\label{rectangular element}
\begin{cases}
F_{T_q}\in (Q_2(\hat{T}_q))^2,\\
\bm{x} = F_{T_q}(\hat{\bm{x}}) = \sum\limits_{i=0}^{8}\bm{a}_i \hat{\varphi}_i(\hat{\bm{x}})
\end{cases}
\end{align}
where $\hat{\varphi}_i$ satisfied that $\hat{\varphi}_i(\hat{\bm{a}}_j)=\delta_{ij}$,
$0\le i,j\le 8$, are the biquadratic interpolation basis functions. If $F_{T_q}$
given above is an injection, then $T_q=F_{T_q}(\hat{T}_q)$ defines a curve edged
quadrilateral element. We define the
iso-parametric $Q_2$-$P_1$ mixed finite element $(T_q, P, \Sigma)$ as follows:
\begin{equation*}
\begin{cases}
T_q=F_{T_q}(\hat{T}_q)\ \text{being a curved quadrilateral element, see Fig~\ref{element qual},}\\
P=\big\{(\bm{u}, p) : T_q\to \mathbb{R}^2\times \mathbb{R} | \bm{u}=\hat{\bm{u}}\comp F_{T_q}^{-1},
\hat{\bm{u}}\in Q_2;\ p=\hat{p}\comp F_{T_q}^{-1}, \hat{p}\in P_1\big\},\\
\Sigma=\big\{\bm{u}(\bm{a}_i), 0\le i\le 8;\
p(\bm{b}_0),\partial_{\hat{x}_1}\hat{p}(\hat{\bm{b}}_0)\comp
F_{T_q}^{-1},\partial_{\hat{x}_2}\hat{p}(\hat{\bm{b}}_0)\comp F_{T_q}^{-1}\big\},
\end{cases}
\end{equation*}

\subsection{The partition of $\Omega_\rho$}

Let $\bm{x}_k$ $(k=1,2,\cdots,K)$ be the center of the $k$-th defect with radius
$\rho_k$ on the reference configuration $\Omega_{\rho}$. The triangulation $\mathscr{T}$
of $\Omega_{\rho}$ consists of two parts: 1) $\mathscr{T}'$ on small circular ring regions
$\cup_{k=1}^K (B_{\delta_k}(\bm{x}_k)\setminus B_{\rho_k}(\bm{x}_k))$ with
$\delta_k>\rho_k$ small; 2) $\mathscr{T}''$ on $\Omega_{\rho} \setminus
\cup_{k=1}^K (B_{\delta_k}(\bm{x}_k)\setminus B_{\rho_k}(\bm{x}_k))$. Each circular ring region
$B_{\delta_k}(\bm{x}_k)\setminus B_{\rho_k}(\bm{x}_k))$ is divided into $M_k$ layers of
circular rings, and the $m$-th layer circular ring is partitioned into $N_m$ evenly spaced
curve edged rectangles with either $N_m=N_{m-1}$ or $2 N_m=N_{m-1}$, where $\{N_m\}_{m=1}^{M_k}$
and the thickness of the layers are given according to the meshing strategy, which approximately
realizes relative error equi-distribution on elastic energy by exploring the relationship between
the error and the energy density \cite{SuLiRectan} (see also \cite{SuLi2018,HuangLi2017}).
The partition $\mathscr{T}'$
on the $m$-th layer circular ring consists of $N_m$ curve edged rectangular elements if
$N_m=N_{m-1}$, otherwise the layer is partitioned into $3N_m$ curve edged triangular
elements by dividing each rectangular element into three triangular elements as shown
in  Fig~\ref{hanging nodes}.
$\mathscr{T}''$ consists of curved or straight edged triangles as shown in Fig~\ref{T''}.

\begin{figure}[H]
    \centering  \includegraphics[width=4.8in]{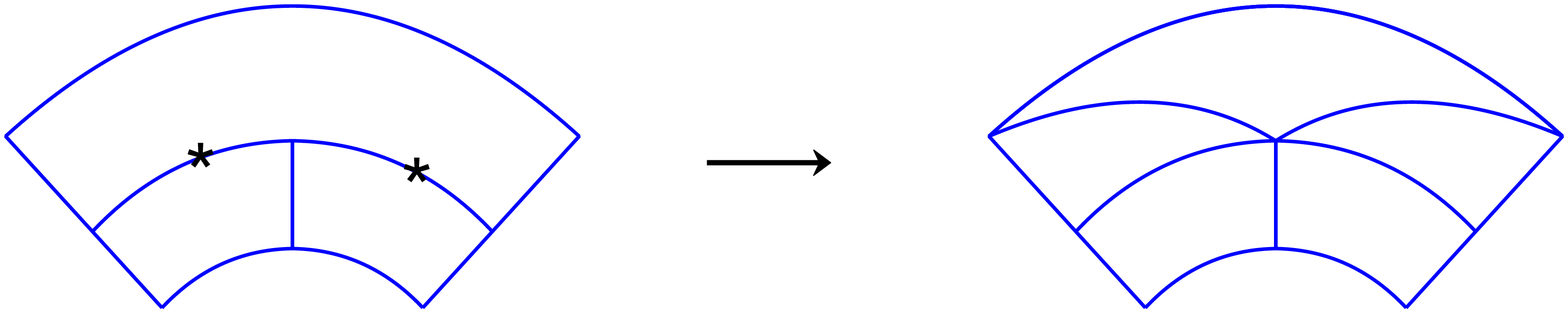}
  \caption{On layers with $2N_m=N_{m-1}$, each rectangle is divided into three triangles.}
  \label{hanging nodes}
\end{figure}

\begin{figure}[H]
\centering \subfigure[$\mathscr{T}'$ near a defet.]{
\label{T'}
    \includegraphics[width=2.2in, height=1.64in]{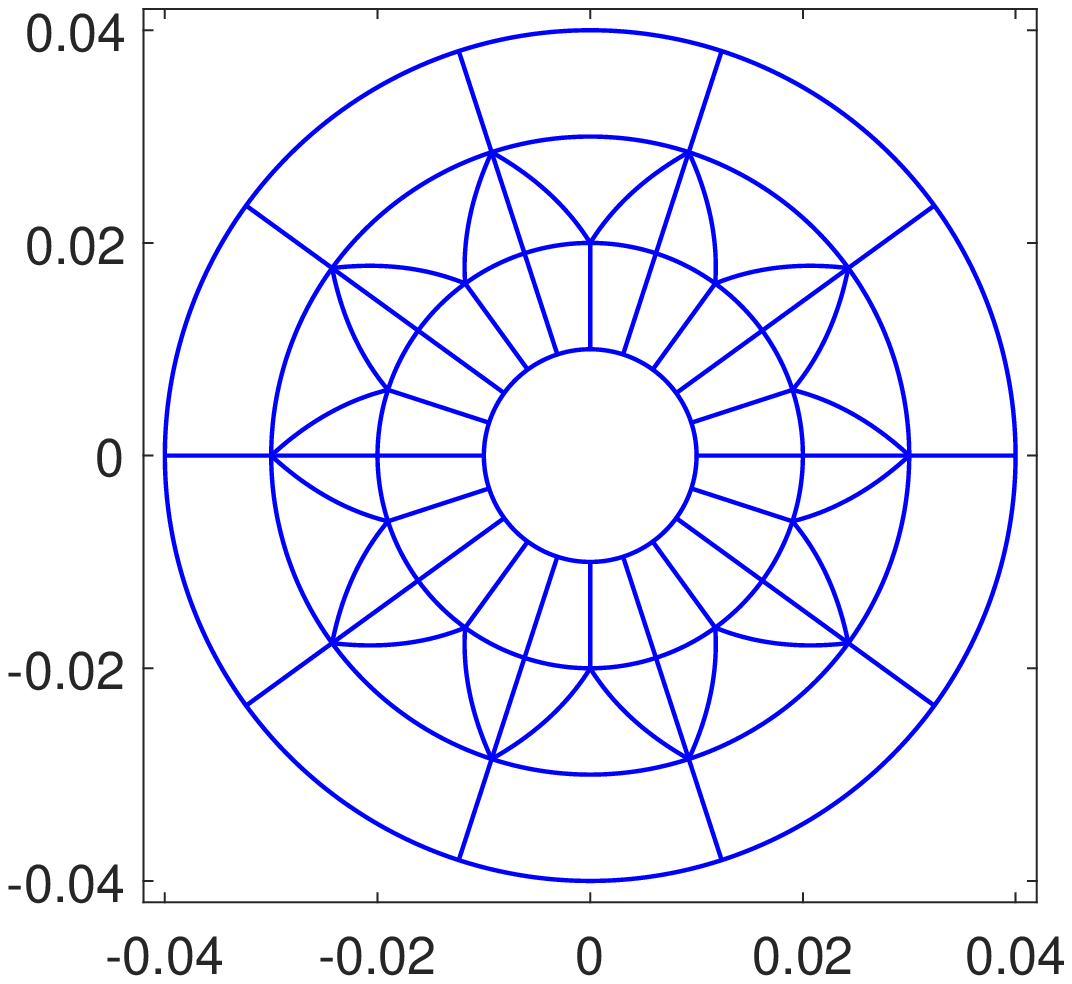}
} \hspace{-12mm}
\centering \subfigure[$\mathscr{T}''$ away from defects.]{
\label{T''}
    \includegraphics[width=2.2in, height=1.64in]{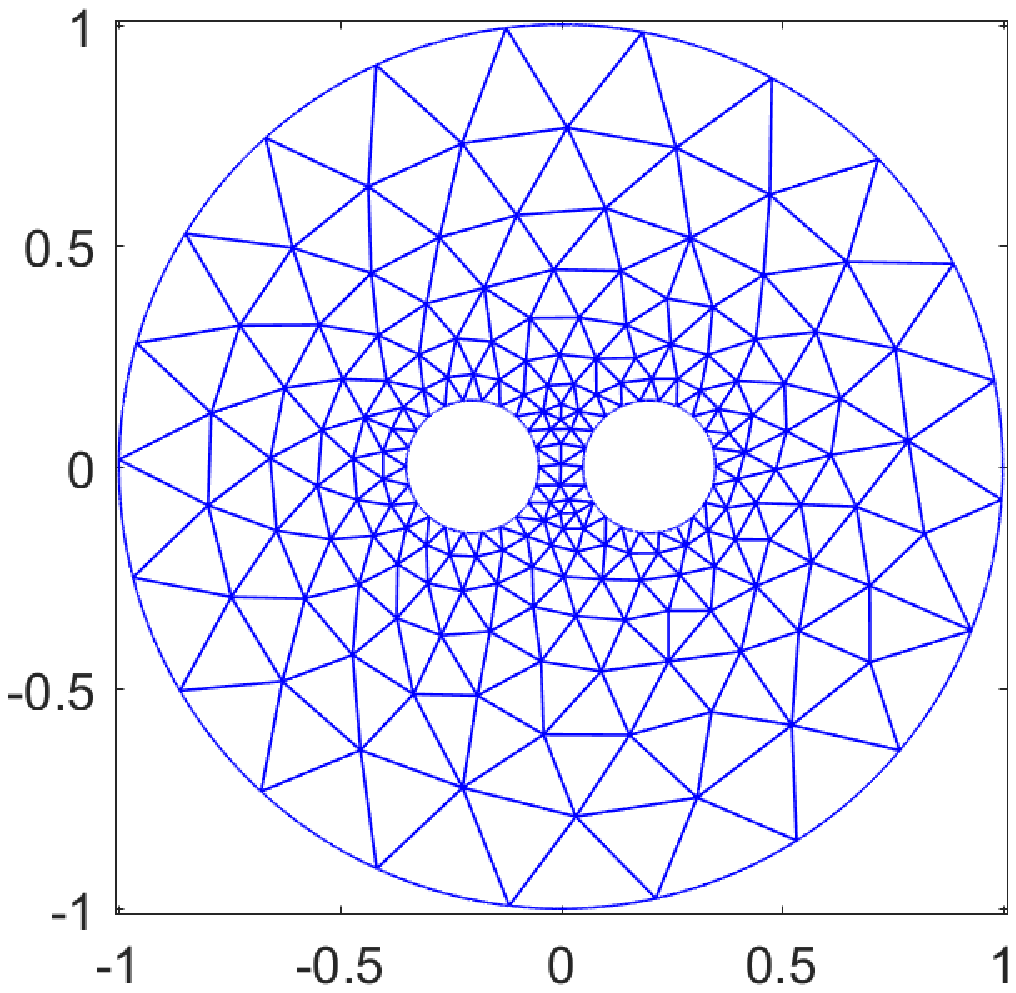}
} \hspace{-12mm}
\centering \subfigure[$\mathscr{T}=\mathscr{T}'+\mathscr{T}''$.]{
\label{T}
    \includegraphics[width=2.2in, height=1.64in]{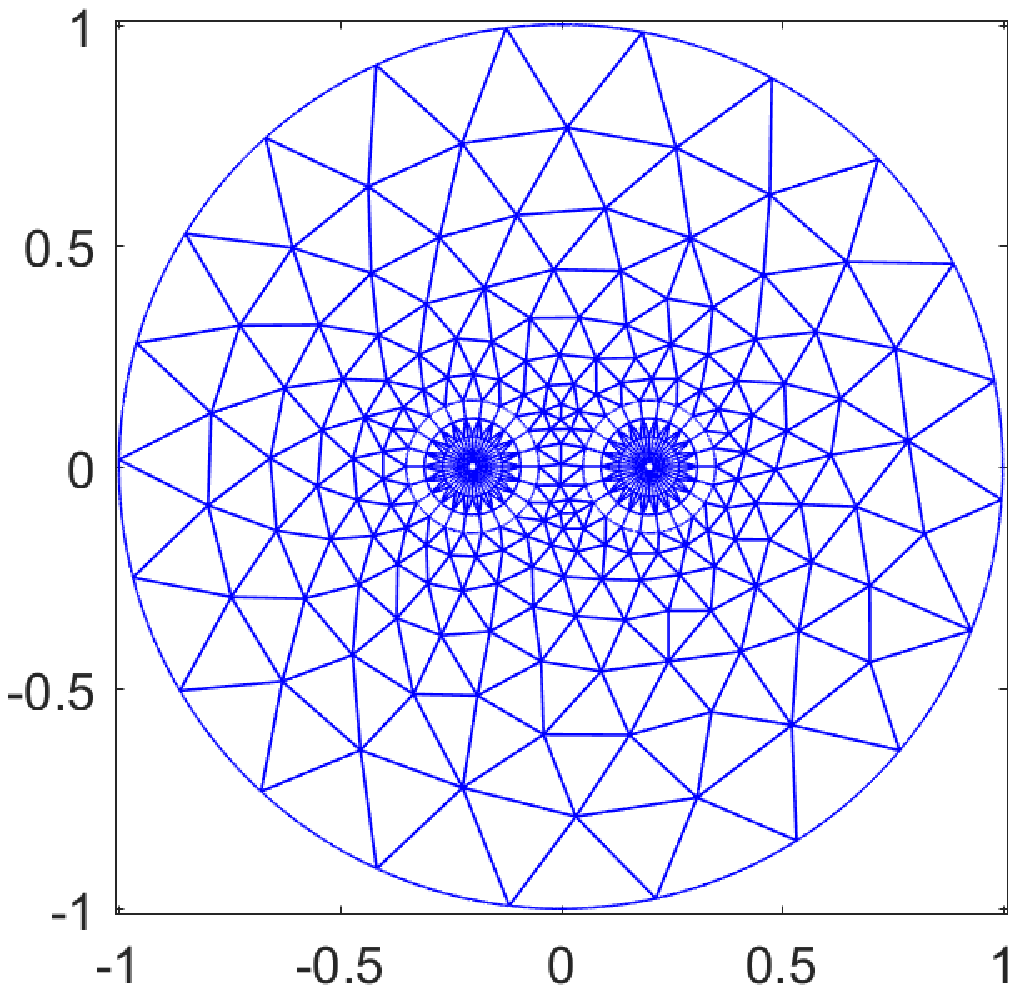}
}
  \caption{Typical partitions $\mathscr{T}'$ (in local coordinate),
  $\mathscr{T}''$ and $\mathscr{T}$.}
  \label{an example of meshing}
\end{figure}

An example of $\mathscr{T}'$ on a circular ring region near a defect is shown in Fig~\ref{T'},
where we have $M=3$, $N_1=20$, $N_2 = N_3 =10$. Notice that, the second layer is divided into
30 triangular elements so that the hanging nodes, denoted by $*$'s in Fig~\ref{hanging nodes},
can be eliminated (see also \cite{SuLi2018}), and for this reason such a layer is called
a conforming layer in contrast to the standard layers consisting of rectangular elements.
An example of $\mathscr{T}''$ on $B_{1}(\bm{0})\setminus (\cup_{k=1}^2 (B_{0.1}(\bm{x}_k)
\setminus B_{0.01}(\bm{x}_k)))$ with
$\bm{x}_1=(-0.2, 0.0)$ and $\bm{x}_2=(0.2, 0.0)$ is shown
in Fig~\ref{T''}, and the final mesh $\mathscr{T}$ produced by $\mathscr{T}'$
and $\mathscr{T}''$ is shown in Fig~\ref{T}.

In what follows below, whenever necessary, the quadrilateral and triangular elements
will be denoted as $T_q$ and $T_t$ respectively.

\subsection{The discrete problem}

We consider to numerically solve the following variational formulation of the
Euler-Lagrange equation \eqref{Euler Lagrange equation-1}-\eqref{Euler Lagrange equation-4}: find
$(\bm{u}, p) \in \mathcal{A}\cap W^{1,\infty}(\Omega_\rho)\times L^2(\Omega_\rho)$, such that
\begin{equation}
\label{weak form of Euler Lagrange equation}
\left\{
\begin{aligned}
\int_{\Omega_{\rho}}\dfrac{\partial W(\nabla\bm{u})}{\nabla\bm{u}}:\nabla\bm{v}-p
\operatorname{cof} \nabla \bm{u}:\nabla \bm{v}\,
\mathrm{d}\bm{x}& = 0,
\quad \forall \bm{v}\in H_{E}^1(\Omega_\rho), \\
\int_{\Omega_{\rho}}q(\det\nabla \bm{u}-1) \,\mathrm{d}\bm{x}&=0,
\ \ \  \forall q\in L^2(\Omega_\rho),
\end{aligned}
\right.
\end{equation}
where $H_E^1(\Omega_\rho) := \{\bm{v}\in H^1(\Omega_\rho):\ \bm{v}|_{\partial_D\Omega_\rho} = \bm{0}\}$.

The finite element trial and test function spaces for the admissible deformation are given as
\begin{align}
\label{discrete admissible set A}
\mathcal{X}_h=\mathcal{A}_h:=& \Big\{\bm{u}_h\in C(\bar{\Omega}_{\rho}):\
\bm{u}_h|_{T_q}\in F_{T_q}(Q_2), \bm{u}_h|_{T_t}\in F_{T_t}(P_2),
\bm{u}_h|_{\partial_D\Omega} = \bm{u}_0 \Big\}, \\
\label{test function}
\mathcal{X}_{h,E}:=& \Big\{\bm{u}_h\in C(\bar{\Omega}_{\rho}):\
\bm{u}_h|_{T_q}\in F_{T_q}(Q_2), \bm{u}_h|_{T_t}\in F_{T_t}(P_2),
\bm{u}_h|_{\partial_D\Omega} = \bm{0} \Big\},
\end{align}
and the finite element trial (and test) function space for the pressure is given as
\begin{align}
& \mathcal{M}_h=\mathcal{P}_h:=\Big\{p_h\in L^2(\bar{\Omega}_{\rho}):\
p_h|_{T}\in F_{T}(P_1)\Big\}.
\label{discrete admissible set P}
\end{align}
The discrete version of
\eqref{weak form of Euler Lagrange equation} is
\begin{equation}
\label{discrete Euler Lagrange equation}
\left\{
\begin{aligned}
\text{Find } (\bm{u}_h,p_h) \in \mathcal{X}_{h} \times \mathcal{M}_h, \,\text{such}
\;\text{that} \qquad\quad & \\ \int_{\Omega_{\rho}}\dfrac{\partial W(\nabla\bm{u}_h)}{\nabla\bm{u}_h}:
\nabla\bm{v}_h-p_h \operatorname{cof} \nabla \bm{u}_h:\nabla \bm{v}_h\,
\mathrm{d}\bm{x}& = 0,
\quad \forall \bm{v}_h\in \mathcal{X}_{h,E}, \\
\int_{\Omega_{\rho}}q_h(\det\nabla \bm{u}_h-1) \,\mathrm{d}\bm{x}&=0,
\ \ \  \forall q_h\in \mathcal{M}_h.
\end{aligned}
\right.
\end{equation}
A damped Newton method is applied to solve this nonlinear system, and in each Newton iteration
step we need to solve the following discrete linear problem:
\begin{equation}
\label{Discrete Newton Method}
\left\{
\begin{aligned}
\text{Find } (\bm{w}_h,p_h) \in \mathcal{X}_{h,E} \times \mathcal{M}_h,\quad & \!\!\!\text{such}
\;\text{that} \\ a(\bm{w}_h,\bm{v}_h; \underline{\bm{u}}_h,\underline{p}_h)+
b(\bm{v}_h,p_h;\bm{\underline{u}}_h) & = f(\bm{v}_h; \underline{\bm{u}}_h,\underline{p}_h),\quad
\forall \bm{v}_h\in \mathcal{X}_{h,E}, \\
b(\bm{w}_h,q_h;\bm{\underline{u}}_h)&=g(q_h;\bm{\underline{u}}_h), \quad \ \ \ \ \
\forall q_h\in \mathcal{M}_h,
\end{aligned}
\right.
\end{equation}
where $\underline{\bm{u}}_h : = \bm{u}_h^k\in\mathcal{A}_h$, $\underline{p}_h := p_h^k\in\mathcal{M}_h$ represent
the approximate solution obtained in the k-th iteration, $(\bm{w}_h, p_h)$ provides a
modifying direction of the (k+1)-th step. An incomplete linear search is conducted so that
the new guess $(\underline{\bm{u}}_h,\underline{p}_h) + \alpha(\bm{w}_h,p_h)$ with
$0<\alpha\le 1$ is orientation preserving and satisfies the regularity condition (H2)
given below. For the energy density $W(\cdot)$ given by \eqref{energy density}, we have
\begin{eqnarray}
\label{a(;)} && \!\!\!\!\!\!\!\!\!\!\!\!\!\!\!\!\!\!\!
a(\bm{w}, \bm{v};\underline{\bm{u}},\underline{p}):=
\int_{\Omega_{\rho}}\left( \mu s(s-2)|\nabla \bm{\underline{u}}|^{s-4}
(\nabla\bm{\underline{u}}: \nabla \bm{w})(\nabla\bm{\underline{u}}:\nabla \bm{v})
\right. \nonumber \\ && \qquad  \left.
\quad + \mu s|\nabla\bm{\underline{u}}|^{s-2}(\nabla \bm{w}:\nabla \bm{v})
+d''(\det\nabla\bm{\underline{u}})(\operatorname{cof}\nabla
\bm{\underline{u}}:\nabla\bm{w})(\operatorname{cof}\nabla\bm{\underline{u}}:
\nabla\bm{v}) \right. \nonumber \\ && \qquad \left.
\quad +(d'(\det\nabla\bm{\underline{u}})-\underline{p})
\operatorname{cof}\nabla\bm{w}:\nabla\bm{v}\right) \,\mathrm{d}\bm{x}, \\
\label{b(;)} && \!\!\!\!\!\!\!\!\!\!\!\!\!\!\!\!\!\!\!  b(\bm{v},
q;\underline{\bm{u}}):=\int_{\Omega_{\rho}} q\operatorname{cof}\nabla
\bm{\underline{u}}:\nabla \bm{v}\,\mathrm{d}\bm{x},\\
\label{f(v)} && \!\!\!\!\!\!\!\!\!\!\!\!\!\!\!\!\!\!\! f(\bm{v};\underline{\bm{u}},\underline{p}):=
 -\int_{\Omega_{\rho}}\left(\mu s|\nabla\underline{\bm{u}}|^{s-2}\nabla\underline{\bm{u}}:\nabla\bm{v}
+(d'(\det\nabla\bm{\underline{u}})-\underline{p})
\operatorname{cof}\nabla\bm{\underline{u}}:\nabla\bm{v}\right)\, \mathrm{d}\bm{x}, \\
\label{g(q)} && \!\!\!\!\!\!\!\!\!\!\!\!\!\!\!\!\!\!\!  g(q;\underline{\bm{u}}):=
-\int_{\Omega_{\rho}} q(\det\nabla\underline{\bm{u}}-1) \,\mathrm{d}\bm{x}.
\end{eqnarray}
In what follows below, if $(\underline{\bm{u}},\underline{p})$ is not directly involved
in the calculation, it will be omitted from the notations
in the functionals defined above. For example, $a(\bm{w},\bm{v};\underline{\bm{u}}, \underline{p})$
will simply be written as $a(\bm{w},\bm{v})$, etc..

To show the stability of the iso-parametric mixed finite element method for the discrete
linear problem \eqref{Discrete Newton Method}, we need to make some basic regularity assumptions:

\begin{description}
\item[{\bf (H1)}] The triangulation $\mathscr{T}$ is regular and $h_T\cong h$, \
$\forall T\in \mathscr{T}$.
\end{description}
\begin{description}
\item[{\bf (H2)}] $0 < \sigma \lesssim \lambda_1(\nabla\underline{\bm{u}}_h) \le
\lambda_2(\nabla\underline{\bm{u}}_h) \lesssim \sigma^{-1}$, and
$0 < c \le \det\nabla\underline{\bm{u}}_h \le C$, $\forall \bm{x} \in T\in\mathscr{T}$,
where $\sigma <1$, $0<c<1$ and $C>1$ are constants.
\item[{\bf (H3)}] For $\underline{\bm{u}}\in W^{1,\infty}(\Omega_\rho)\cap H^1(\Omega_\rho)$
satisfying (H2), $b(\bm{v},q;\underline{\bm{u}})$ satisfies inf-sup condition, i.e.
there exists a constant $\beta>0$ such that
\begin{equation}
\label{inf-sup condition}
\begin{aligned}
\sup_{\bm{v}\in H_E^1(\Omega_\rho)}\dfrac{b(\bm{v},q)}{\|\bm{v}\|_{1,2,\Omega_\rho}}
\ge \beta\|q\|_{0,2,\Omega_\rho}, \quad \forall q \in L^2(\Omega_\rho).
\end{aligned}
\end{equation}
\end{description}

Here and throughout the paper, $X \cong Y$, or equivalently $Y\lesssim X \lesssim Y$, means that
$C^{-1}Y \le X \le CY$ holds for a generic constant $C \ge 1$ independent of $\rho$, $T$ and $h$.

\begin{rem}
By the standard scaling argument, hypothesis (H1) guarantees that
\begin{equation}\label{scaling}
|\hat{\bm{v}}|_{\gamma,2,\hat{T}}\cong h_T^{\gamma-1}|\bm{v}|_{\gamma,2,T}, \;\; \gamma=0,1, \quad
\forall T \in \mathscr{T} \; \text{and} \;\;\forall \bm{v}\in H^1(T),
\end{equation}
which is a very important relation in many interpolation error estimates. In fact,
(H2) generally holds for a physically meaningful discrete cavitation deformation
$\underline{\bm{u}}_h$.\end{rem}
\begin{rem}
Hypothesis (H3) guarantees the solvability of the linear problem
\begin{equation}
\label{Continuous Newton Method}
\left\{
\begin{aligned}
\text{Find } (\bm{w},p) \in H_E^1(\Omega_\rho) \times L^2(\Omega_\rho),\quad & \!\!\!\text{such}
\;\text{that} \\ a(\bm{w},\bm{v}; \underline{\bm{u}},\underline{p})+
b(\bm{v},p;\bm{\underline{u}}) & = f(\bm{v}; \underline{\bm{u}},\underline{p}),\quad
\forall \bm{v}\in H_E^1(\Omega_\rho), \\
b(\bm{w},q;\bm{\underline{u}})&=g(q;\bm{\underline{u}}), \quad \ \ \ \
\forall q\in L^2(\Omega_\rho),
\end{aligned}
\right.
\end{equation}
which is the continuous counterpart of \eqref{Discrete Newton Method} with respect to the weak form
of Euler-Lagrange equation \eqref{weak form of Euler Lagrange equation}.
(H3) is also a basic condition for Fortin Criterion, which is essential for our stability
analysis. It is worth mentioning here that \eqref{inf-sup condition} does hold
if certain additional regularity condition, say
$\nabla\underline{\bm{u}} \in C^1(\overline{\Omega}_{\rho})$, is satisfied (see\cite{Dobrowolski}) .
\end{rem}

\section{Stability analysis of the method}

Under hypothesis (H1) for $\mathscr{T}$, (H2) for $\underline{\bm{u}}_h$ and (H3), we will prove
that the iso-parametric mixed finite element method for the problem \eqref{Discrete Newton Method}
is locking-free. More precisely, there exists a constant $\beta>0$ independent of $h$,
such that the discrete inf-sup condition
\begin{equation}
\label{LBB condition}
\sup_{\bm{v}_h\in \mathcal{X}_h}\dfrac{b(\bm{v}_h,q_h;\bm{\underline{u}}_h)}{\Vert
\bm{v}_h\Vert _{1,2,\Omega_\rho}}\ge \beta\Vert q_h\Vert _{0,2,\Omega_{\rho}},
\quad  \forall q_h\in \mathcal{M}_h
\end{equation}
holds.
For the convenience of the readers and the integrity of this paper, we present the
stability analysis below, even though it is standard, based on the famous Fortin Criterion
\cite{Fortion1991} (see Lemma~\ref{Fortin Criterion}) and a two steps construction process
\cite{Fortion1991} (see Lemma~\ref{2 steps construction}), and follows the similar lines as
in \cite{HuangLi2017}.
\begin{Lemma}
\label{Fortin Criterion}
(see Fortin Criterion \cite{Fortion1991}) Let $b(\bm{v},q;\underline{\bm{u}}_h)$ satisfy
the inf-sup condition \eqref{inf-sup condition}. Then, the LBB consition \eqref{LBB condition}
holds with a constant $\beta$ independent of $h$ if and only if there exists an operator
$\Pi_h\in \mathscr{L}(H_E^1(\Omega_\rho),\mathcal{X}_{h,E})$ and a constant $c>0$ independent
of $h$ such that
\begin{equation}
\label{the operator}
\left\{
\begin{aligned}
 &b(\bm{v}-\Pi_h\bm{v},q_h;\underline{\bm{u}}_h)=0,\qquad  \forall q_h\in \mathcal{M}_h,\
 \forall \bm{v}\in H_E^1(\Omega_\rho), \\
 &\Vert \Pi_h\bm{v}\Vert _{1,2,\Omega_{\rho}}\le c\Vert \bm{v}
 \Vert _{1,2,\Omega_{\rho}},\quad \; \forall \bm{v}\in H_E^1(\Omega_\rho).
\end{aligned}
\right.
\end{equation}
\end{Lemma}

\begin{Lemma}
\label{2 steps construction}
Let $\Pi_1\in \mathscr{L}(H_E^1(\Omega_\rho),\mathcal{X}_{h,E})$ and $\Pi_2\in\mathscr{L}(H_E^1(\Omega_\rho),\mathcal{X}_{h,E})$
be such that
\begin{equation}
\label{operator construction}
\left\{
\begin{aligned}
&\Vert \Pi_1\bm{v}\Vert _{1,2,\Omega_\rho}\le c_1\Vert \bm{v}\Vert _{1,2,\Omega_\rho},
\quad \forall \bm{v}\in H_E^1(\Omega_{\rho}),\\
&\Vert \Pi_2(I-\Pi_1)\bm{v}\Vert _{1,2,\Omega_\rho}\le c_2\Vert
\bm{v}\Vert _{1,2,\Omega_\rho},\quad \forall \bm{v}\in H_E^1(\Omega_\rho),\\
&b(\bm{v}-\Pi_2\bm{v},q_h;\bm{\underline{u}}_h)=0,\quad  \forall \bm{v}\in H_E^1(\Omega_\rho),\
\forall q_h\in \mathcal{M}_h.
\end{aligned}
\right.
\end{equation}
Set $\Pi_h\bm{v}=\Pi_1\bm{v}+\Pi_2(\bm{v}-\Pi_1\bm{v})$, then $\Pi_h\in
\mathscr{L}(H_E^1(\Omega_\rho),\mathcal{X}_{h,E})$ satisfies \eqref{the operator}.
\end{Lemma}

\emph{\textbf{Step 1}}. The construction of $\Pi_1\in \mathscr{L}(H_E^1(\Omega_\rho),\mathcal{X}_h)$.

Let $(\bar{\mathcal{X}}_h,\bar{\mathcal{M}}_h)$ be given by
\begin{equation}
\label{sub space}
\begin{cases}
\bar{\mathcal{X}}_h=\{ \bm{v}_h \in \mathcal{X}_h : \ \bm{v}_h|_{T_q} \in
F_{T_q}(Q_1)\oplus\text{span}\{\bm{q}_1,\bm{q}_2,\bm{q}_3,\bm{q}_4\}, \
\bm{v}_h|_{T_t} \in F_{T_t}(P_2) \}, \\
\bar{\mathcal{M}}_h=\{ q_h \in \mathcal{M}_h : \ q_h|_{T}\in F_T(P_0)\},
\end{cases}
\end{equation}
where $\{\bm{q}_{i}\}_{i=1}^{4}$
are the edge bubble functions with respect to the edges $\{e_{i}\}_{i=1}^{4}$ of $T_q$.
For example, denote $\bm{x}=(x_1,x_2)$ and
$(\hat{x}_1,\hat{x}_2)=\hat{\bm{x}}=F_{T_q}^{-1}(\bm{x})$, then
$$
\bm{q}_1(\bm{x})=(\hat{q}_1\comp F_{T_q}^{-1}(\hat{\bm{x}}))\ \bm{n}_1(F_{T_q}(-1,\hat{x}_2)),
$$
where $\hat{q}_1=(1-\hat{x}_2^2)(1-\hat{x}_1)$ and $\bm{n}_1$ is the unit out normal of the
edge $e_1$. Obviously $\bm{q}_1(\bm{x}) =0$, $\forall \bm{x} \in \partial T_q \setminus e_1$.
The formulae for $\{\bm{q}_{i}\}_{i=2}^{4}$ are similarly defined. In particular,
we notice that $\{\bm{q}_{i}\}_{i=1}^{4}$ have null tangential components on the edges
of $T_q$. On the other hand, let $\{\bm{p}_i(\bm{x})\}_{i=1}^3$ be the edge bubble functions
with respect to the edges of a triangular element $T_t$, defined in a similar way as
$\{\bm{q}_i\}_{i=1}^4$, then we have
$F_{T_t}(P_2) = F_{T_t}(P_1)\oplus\text{span}\{\bm{p}_1,\bm{p}_2,\bm{p}_3\}$.

Firstly, let $\bar{\Pi}_1: H_E^1(\Omega_\rho) \rightarrow
\mathcal{X}_h\cap H_E^1(\Omega_\rho)$ be the Cl\'{e}ment interpolation operator, then,
under the hypothesis (H1) and by the standard scaling argument (see for example Corollary 2.1
on page 106 in \cite{Fortion1991}), one has
\begin{equation}\label{Clement}
\sum_{T \in \mathscr{T}}h_T^{2\gamma-2}|\bm{v}-\bar{\Pi}_1\bm{v}|_{\gamma,2,T}^2\lesssim
|\bm{v}|_{1,2,\Omega_\rho}^2,\quad \gamma = 0,1.
\end{equation}

Next, let $\bar{\Pi}_2: H_E^1(\Omega_\rho) \to \mathcal{X}_h\cap H_E^1(\Omega_\rho)$
be uniquely determined by
\begin{equation}
\label{operator2}
\left\{
\begin{aligned}
&\bar{\Pi}_2\bm{v}|_{T_t}\in \text{span}\{\bm{p}_1,\bm{p}_2,\bm{p}_3\},\\
&\bar{\Pi}_2\bm{v}|_{T_q}\in \text{span}\{\bm{q}_1,\bm{q}_2,\bm{q}_3,\bm{q}_4\},\\
&\int_{e_i}\operatorname{cof}\nabla \underline{\bm{u}}_h^{\rm T}
(\bar{\Pi}_2\bm{v}-\bm{v})\cdot \bm{n}_i\,\mathrm{d}s =0,
\ \forall  e_i\in \partial {T},\ \forall T\in\mathscr{T},
\end{aligned}
\right.
\end{equation}
then, as $\operatorname{div}(\operatorname{cof}\nabla\underline{\bm{u}}_h|_T)=0$, one has
\begin{equation}
\label{third condition}
\begin{aligned}
\int_T \operatorname{cof}\nabla\underline{\bm{u}}_h:\nabla(\bar{\Pi}_2\bm{v}-\bm{v})\,\mathrm{d}\bm{x}
=\int_{\partial T}(\operatorname{cof}\nabla\underline{\bm{u}}_h^{\rm T}
(\bar{\Pi}_2\bm{v}-\bm{v}))\cdot\bm{n}\,\mathrm{d}s=0,\ \forall T\in\mathscr{T}.
\end{aligned}
\end{equation}

Now, define $\Pi_1 \in \mathscr{L}(H_E^1(\Omega_\rho),\mathcal{X}_h)$ as
$\Pi_1 \bm{v} \triangleq \bar{\Pi}_h\bm{v}=\bar{\Pi}_1\bm{v}+\bar{\Pi}_2
(\bm{v}-\bar{\Pi}_1\bm{v})$, $\forall \bm{v} \in H_E^1(\Omega_\rho)$.

\vskip 2mm
\emph{\textbf{Step 2.}} The construction of $\Pi_2 \in \mathscr{L}(H_E^1(\Omega_\rho),\mathcal{X}_h)$.

Introduce a bubble function space on $\mathscr{T}$ by defining
\begin{equation}\label{bubble}
\mathcal{B}_h=\{ \bm{b}\in C(\bar{\Omega}_{\rho};\mathbb{R}^2) :
\bm{b}|_T=\hat{\bm{b}} \comp F_T^{-1}\},
\end{equation}
where $\hat{\bm{b}}(\hat{\bm{x}})=(b_1(1-\hat{x}_1^2)(1-\hat{x}_2^2),
b_2(1-\hat{x}_1^2)(1-\hat{x}_2^2))$ if $T \in \mathscr{T}$ is a quadrilateral
element, and $\hat{\bm{b}}(\hat{\bm{x}})=(b_1\hat{\lambda}_1(\hat{\bm{x}})
\hat{\lambda}_2(\hat{\bm{x}})\hat{\lambda}_3(\hat{\bm{x}}), b_2\hat{\lambda}_1(\hat{\bm{x}})
\hat{\lambda}_2(\hat{\bm{x}})\hat{\lambda}_3(\hat{\bm{x}}))$ if $T \in \mathscr{T}$ is a
triangular element, and $b_1$, $b_2\in \mathbb{R}$. Define
$\Pi_2 : \{\bm{v}\in H_E^1(\Omega_\rho):\ \int_T \operatorname{cof}\nabla\underline{\bm{u}}_h:
\nabla\bm{v}\,\mathrm{d}\bm{x}=0,\forall T\in\mathscr{T}\}\rightarrow \mathcal{B}_h$ as
the unique solution of the linear system
\begin{equation}
\label{operator2 construction}
\int_{T} \operatorname{cof}\nabla\underline{\bm{u}}_h:\nabla(\Pi_2\bm{v}-\bm{v})\ q_h
\,\mathrm{d}\bm{x}=0,\ \ \forall q_h\in F_T(P_1), \ \
\forall T \in \mathscr{T}.
\end{equation}
Notice that \eqref{operator2 construction} naturally holds for $\forall q_h\in F_T(P_0)$.
This is because $\Pi_2\bm{v}|_{\partial T} =0$ and
$\operatorname{div}(\operatorname{cof}\nabla\bm{\underline{u}}_h|_T)=0$,
hence, by the divergence theorem, one has
\begin{equation}\label{Pi v null}
\int_{T}\operatorname{cof}\nabla\underline{\bm{u}}_h:\nabla\Pi_2\bm{v} \,
\mathrm{d}\bm{x}=\int_{\partial T}(\operatorname{cof}\nabla
\underline{\bm{u}}_h)\bm{n}\cdot \Pi_2\bm{v}
\,\mathrm{d}s-\int_{T}\operatorname{div}(\operatorname{cof}\nabla\bm{\underline{u}}_h)
\cdot \Pi_2\bm{v}\, \mathrm{d}\bm{x} = 0.
\end{equation}

\begin{Lemma}\label{Pi_1 lemma}
Let $\mathscr{T}$ and $\underline{\bm{u}}_h$ satisfy hypothesis (H1) and (H2) respectively.
Then $\Pi_1$ defined in step~1 satisfies $\Pi_1\in\mathscr{L}(H_E^1(\Omega_\rho),\bar{\mathcal{X}}_h)$
and
\begin{equation}
\left\{
\begin{aligned}
\label{operator1 property}
&\Vert \Pi_1\bm{v}\Vert _{1,2,\Omega_\rho}\lesssim\dfrac{1}{\sigma^2}\Vert \bm{v}
\Vert _{1,2,\Omega_\rho},\quad \forall \bm{v}\in H_E^1(\Omega_\rho),\\
&\int_{T} \operatorname{cof}\nabla \underline{\bm{u}}_h:
\nabla(\Pi_1\bm{v}-\bm{v})\,\mathrm{d}\bm{x}=0,\quad \forall \bm{v}\in H_E^1(\Omega_\rho),
\ \forall T\in \mathscr{T}.
\end{aligned}
\right.
\end{equation}
\end{Lemma}

\begin{proof}
Since $\bar{\Pi}_2\in \mathscr{L}(H_E^1(\Omega_\rho),\mathcal{\bar{X}}_h)$ is defined through
\eqref{operator2}, $\bar{\Pi}_2$ on $T_t$ can be explicitly expressed as
$\bar{\Pi}_2\bm{v}|_{T_t}=\sum\limits_{i=1}^3\alpha_i(\bm{v})\bm{p}_i$, which yields
\begin{equation}
\label{alpha i  definition}
\begin{aligned}
\alpha_i=\Big[\int_{e_i}(\operatorname{cof}\nabla\underline{\bm{u}}_h^T\bm{v})
\cdot\bm{n}_i\,\mathrm{d}s\Big]\Big/\Big[\int_{e_i}(\operatorname{cof}
\nabla\underline{\bm{u}}_h^T\bm{p}_i)\cdot \bm{n}_i\, \mathrm{d}s\Big],\quad i =1,2,3.
\end{aligned}
\end{equation}
Noticing that (H2) implies $\sigma\lesssim \lambda_1(\operatorname{cof}\nabla\bm{u}_h)
\lesssim \lambda_2(\operatorname{cof}\nabla \underline{\bm{u}}_h)
\lesssim \sigma^{-1}$, thus one has
\begin{equation}\label{denominator}
\begin{aligned}
\Big|\int_{e_i}(\operatorname{cof}\nabla\underline{\bm{u}}_h^T\bm{p}_i)\cdot\bm{n}_i
\,\mathrm{d}s\Big| =\Big|\int_{e_i}\bm{p}_i\cdot(\operatorname{cof}\nabla
\underline{\bm{u}}_h\bm{n}_i) \,\mathrm{d}s\Big|
\gtrsim \sigma h_{T_t}\int_{\hat{e}_i}\hat{p}_i \,\mathrm{d}\hat{s},
\end{aligned}
\end{equation}
and in addition, by the trace theorem,
\begin{equation}\label{numerator}
\Big|\int_{e_i}(\operatorname{cof}\nabla\underline{\bm{u}}_h^T\bm{v})\cdot \bm{n}_i
\, \mathrm{d}s\Big|\lesssim \lambda_2(\nabla\underline{\bm{u}}_h)
\int_{e_i}|\bm{v}|\,\mathrm{d}s
\cong\dfrac{h_{T_t}}{\sigma}\int_{\hat{e}_i}|\hat{\bm{v}}|
\,\mathrm{d}\hat{s}\lesssim \dfrac{h_{T_t}}{\sigma}\|\hat{\bm{v}}\|_{1,2,\hat{T}_t},
\end{equation}
Therefore, it follows from \eqref{alpha i  definition}-\eqref{numerator} that
\begin{equation}
\label{alpha i}
\begin{aligned}
|\alpha_i|\lesssim \dfrac{1}{\sigma^2}\|\hat{\bm{v}}\|_{1,2,\hat{T}_t},\quad i=1,2,3.
\end{aligned}
\end{equation}

Hence, by the standard scaling argument, one has
\begin{equation}
\label{bar pi2 sim-H1}
|\bar{\Pi}_2\bm{v}|^2_{1,2,T_t}=\Big|\sum_{i=1}^3\alpha_i\bm{p}_i\Big|^2_{1,2,T_t}
\lesssim \dfrac{1}{\sigma^4}(h_{T_t}^{-2}\|\bm{v}\|^2_{0,2,T_t}+|\bm{v}|^2_{1,2,T_t}),\;\;
\forall T_t \in \mathscr{T}.
\end{equation}
The result for $\bar{\Pi}_2$ on $T_q$ has the same form and can be obtained in the same way.
Thus, it follows from \eqref{Clement} that
\begin{equation*}
\begin{aligned}
|\Pi_1\bm{v}|^2_{1,2,\Omega_\rho}
\lesssim &|\bar{\Pi}_1\bm{v}|^2_{1,2,\Omega_\rho}
+\sum_{T}|\bar{\Pi}_2(\bm{v}-\bar{\Pi}_1\bm{v})|^2_{1,2,T}\\ \lesssim
&|\bar{\Pi}_1\bm{v}|^2_{1,2,\Omega_\rho} +\sum_{T}\dfrac{1}{\sigma^4}(h_T^{-2}
\|\bm{v}-\bar{\Pi}_1\bm{v}\|^2_{0,2,T}+|\bm{v}-\bar{\Pi}_1\bm{v}|^2_{1,2,T})\lesssim
\dfrac{1}{\sigma^4}|\bm{v}|_{1,2,\Omega_\rho}^2.
\end{aligned}
\end{equation*}
This together with the Poincar\'{e}-Friedrichs inequality implies that the
inequality in \eqref{operator1 property} holds, since $\partial_D\Omega_\rho \neq \varnothing$.
In addition, by \eqref{third condition}, we have, for all $\bm{v}\in H_E^1(\Omega_\rho)$,
\begin{equation*}
\int_T \operatorname{cof}\nabla \underline{\bm{u}}_h:
\nabla(\Pi_1\bm{v}-\bm{v})\,\mathrm{d}\bm{x}=\int_T \operatorname{cof}\nabla \underline{\bm{u}}_h:
\nabla\big(\bar{\Pi}_2
(\bm{v}-\bar{\Pi}_1\bm{v})-(\bm{v}-\bar{\Pi}_1\bm{v})\big)\,\mathrm{d}\bm{x}=0.
\end{equation*}
This completes the proof of the lemma.
\end{proof}

\begin{Theorem} Let hypotheses (H1)-(H3) hold, and $(\mathcal{X}_{h,E},\mathcal{M}_h)$ be
given by \eqref{test function} and \eqref{discrete admissible set P}.
Then, there exists a constant $\beta>0$ independent of $h$ such that
$b(\bm{v}_h,q_h;\underline{\bm{u}}_h)$ satisfies the LBB condition \eqref{LBB condition}.
\end{Theorem}
\begin{proof}
According to Lemma~\ref{2 steps construction}, we only need to show
\eqref{operator construction}$_2$, i.e.,
$\Vert \Pi_2(I-\Pi_1)\bm{v}\Vert_{1,2,\Omega_\rho}\le c_2 \Vert \bm{v}\Vert _{1,2,\Omega_\rho}$,
$\forall \bm{v}\in H_E^1(\Omega_\rho)$, since \eqref{operator construction}$_1$ is a conclusion of
Lemma~\ref{Pi_1 lemma}; and \eqref{operator construction}$_3$ follows as a consequence of the
definition of $\Pi_2$ (see \eqref{operator2 construction}).

Recall that $\Pi_2\bm{v}\in \mathcal{B}_h$ (see \eqref{bubble}) and
$\operatorname{div}(\operatorname{cof}\nabla\bm{\underline{u}}_h|_{T})=0$,
$\forall T \in \mathscr{T}$, by a change of integral variables and the integral by parts,
\eqref{operator2 construction} can be rewritten as
\begin{equation}
\label{rewrite pi2 hat}
\int_{\hat{T}}\widehat{\Pi_2\bm{v}}\cdot \big(\operatorname{cof}\nabla_{\hat{x}}
\underline{\bm{\hat{u}}}_h\nabla_{\hat{x}}\hat{q}_h\big) \, \mathrm{d}\bm{\hat{x}}
=\int_{\hat{T}} \hat{q}_h \operatorname{cof}\nabla_{\hat{x}} \underline{\bm{\hat{u}}}_h:
\nabla_{\hat{x}}\bm{\hat{v}} \,\mathrm{d}\bm{\hat{x}}, \;\;
\forall \hat{q}_h\in P_1(\hat{T})\setminus P_0(\hat{T}),
\end{equation}
where $\nabla_{\hat{x}}:=(\partial_{\hat{x}_1}, \partial_{\hat{x}_2})$. Taking $T_t$ as an example, we can write $\widehat{\Pi_2\bm{v}}(\hat{\bm{x}})$ on $T_t$ explicitly as
$\widehat{\Pi_2\bm{v}}(\hat{\bm{x}})|_{T_t}= (\alpha_1\hat{\lambda}_1\hat{\lambda}_2\hat{\lambda}_3,
\alpha_2\hat{\lambda}_1\hat{\lambda}_2\hat{\lambda}_3)$ with
\begin{equation}
\label{bm alpha}
\bm{\alpha} = \begin{pmatrix} \alpha_1 \\ \alpha_2 \end{pmatrix} =
\left(\int_{\hat{T}_t}\hat{b}\operatorname{cof}\nabla_{\hat{x}}\underline{\hat{\bm{u}}}_h
\,\mathrm{d}\bm{\hat{x}}\right)^{-1} \begin{pmatrix} \int_{\hat{T}_t}
\operatorname{cof}\nabla_{\hat{x}}\underline{\hat{\bm{u}}}_h: \nabla_{\hat{x}}\hat{\bm{v}}
\ \hat{x}_1\,\mathrm{d}\bm{\hat{x}} \\
\int_{\hat{T}_t}\operatorname{cof}\nabla_{\hat{x}}\underline{\hat{\bm{u}}}_h:
\nabla_{\hat{x}}\hat{\bm{v}}\ \hat{x}_2\,\mathrm{d}\bm{\hat{x}}
\end{pmatrix},
\end{equation}
where $\hat{b}=\hat{\lambda}_1\hat{\lambda}_2\hat{\lambda}_3$. Again, since (H2) implies
$\sigma\lesssim \lambda_1(\operatorname{cof}\nabla\bm{u}_h)
\lesssim \lambda_2(\operatorname{cof}\nabla \underline{\bm{u}}_h)
\lesssim \sigma^{-1}$,
it follows from the H\"{o}lder inequality that
\begin{equation}
\Big|\int_{\hat{T}_t}\operatorname{cof}\nabla_{\hat{x}}\underline{\hat{\bm{u}}}_h:
\nabla_{\hat{x}}\bm{\hat{v}}\ \hat{x}_i\,\mathrm{d}\bm{\hat{x}}\Big|
\lesssim \dfrac{h_{T_t}}{\sigma}|\hat{\bm{v}}|_{1,2,\hat{T}_t}\|\hat{x}_i\|_{0,2,\hat{T}_t}
\lesssim \dfrac{h_{T_t}}{\sigma}|\hat{\bm{v}}|_{1,2,\hat{T}_t},\quad i=1,2.
\end{equation}
On the other hand, noticing that $\underline{\bm{u}}_h|_{T_t}\in P_2^+(\hat{T})$,
by direct calculations (similar to that in the proof of Theorem 3.1 in \cite{SuLiRectan}) and (H2),
one has
\begin{equation}
\det\Big(\int_{\hat{T}_t}\hat{b} \operatorname{cof}\nabla_{\hat{x}}\underline{\hat{\bm{u}}}_h(\bm{x})
\,\mathrm{d}\bm{\hat{x}}\Big)\cong\det\nabla\underline{\bm{u}}_h(\bm{a}_{123})h_{T_t}^2\cong h_{T_t}^2.
\end{equation}
Thus, again by (H2),
\begin{equation}
\label{coefficient matrix}
\Big|\big(\int_{\hat{T}_t}\hat{b}\operatorname{cof}\nabla_{\hat{x}}\underline{\hat{\bm{u}}}_h
\, \mathrm{d}\bm{\hat{x}}\big)^{-1}\Big|\cong h_{T_t}^{-2}\Big|\int_{\hat{T_t}}\hat{b}
\nabla_{\hat{x}}\underline{\hat{\bm{u}}}_h\,\mathrm{d}\bm{\hat{x}}\Big|
\lesssim h_{T_t}^{-2} \|\hat{b}\|_{0,2,\hat{T}_t}\|\nabla_{\hat{x}}\underline{\hat{\bm{u}}}_h
\|_{0,2,\hat{T}_t}\lesssim \dfrac{1}{\sigma h_{T_t}}.
\end{equation}
Therefore, it follows from \eqref{bm alpha}-\eqref{coefficient matrix} and the standard
scaling argument that
\begin{equation}
\label{pi2 bubble}
|\Pi_2\bm{v}|_{1,2,T_t}\cong |\widehat{\Pi}_2\bm{v}|_{1,2,\hat{T}_t}\cong |\bm{\alpha}|\lesssim
\dfrac{1}{\sigma^2}|\hat{\bm{v}}|_{1,2,\hat{T}_t}\cong \dfrac{1}{\sigma^2}|\bm{v}|_{1,2,T_t}.
\end{equation}
$\Pi_2\bm{v}$ on $T_q$ has exactly the same result as \eqref{pi2 bubble}
(see also \cite{HuangLi2017}). Finally, by $\eqref{operator1 property}_1$,
\eqref{pi2 bubble}, and Poincar\'{e}-Friedrichs inequality, we have
\begin{equation}
\label{beta magnitude}
\|\Pi_2(I-\Pi_1)\bm{v}\|_{1,2,\Omega_\rho}\lesssim \dfrac{1}{\sigma^2}\|(I-\Pi_1)\bm{v}
\|_{1,2,\Omega_\sigma} \lesssim\dfrac{1}{\sigma^4}\|\bm{v}\|_{1,2,\Omega_\rho},
\  \forall \bm{v}\in H_E^1(\Omega_\rho),
\end{equation}
and complete the proof of the theorem.
\end{proof}

\section{Convergence analysis of the method}

The framework for the convergence analysis of the finite element cavitation solutions
of \eqref{discrete Euler Lagrange equation} is the same as established in
\cite{HuangLi2017} for the DP-Q2-P1 finite element cavitation solutions. However, for the
integrity of the paper and convenience of the readers, the complete analysis is presented
below.

The cavitation solution $\tilde{\bm{u}}$ is assumed to be an absolute energy minimizer
of $E(\cdot)$ in $\mathcal{A}_I$ (see \eqref{functional} \eqref{admissible set incompressible})
and is assumed to have the following regularity:
$\bm{u}\in C^4(\Omega_\rho;\mathbb{R}^2)$ and in a neighborhood of each defect, expressed in
the local polar coordinate system $\bm{u}(\bm{x})=(r\cos\phi,r\sin\phi)$ with
$r=r(R,\theta)$, $\phi=\phi(R,\theta)$,
\begin{equation}
\label{UUpsilon}
\big|\dfrac{\partial\bm{u}}{\partial\theta}\big|\cong 1 \;\; \text{and} \;\;
 \big|\frac{\partial^{i+j}r}{
 \partial R^i\partial \theta^j}\big|\le \Upsilon,\;
\big|\frac{\partial^{i+j}\phi}{\partial R^i\partial\theta^j}\big|\le \Upsilon, \;\;
\forall i,j\ge 0, \; i+j\le 4,
\end{equation}
where $\Upsilon$ is a constant independent of the initial defect size $\rho$.
Denote
\begin{equation*}
 \mathcal{U}(\Upsilon)=\Big\{\bm{u}\in C^4(\Omega_\rho;\mathbb{R}^2)
 :\bm{u}(\bm{x})\ \text{satisfies \eqref{UUpsilon} in a neighborhood of each defect}\Big\}.
\end{equation*}

Let $\mathscr{T}=\mathscr{T}' + \mathscr{T}''$ (see Fig~\ref{T}) be a regular triangulation of
$\overline{\Omega}_{\rho}$ satisfying (H1) with mesh size $h$. Let $\epsilon$ and
$\tau$ denote respectively the inner radius and the thickness of a circular ring layer
$B_{\epsilon+\tau}(\bm{x}_k)\setminus B_{\epsilon}(\bm{x}_k)$, produced by $\mathscr{T}'$ in
a neighborhood of a defect as shown in Fig~\ref{T'}, and let $N$ be the number
of the subdivision on the circular direction of the layer.
Let $\Pi^2_h: \mathcal{A} \cap C(\overline{\Omega}_{\rho}; \mathbb{R}^2) \to \mathcal{A}_h$
(see \eqref{discrete admissible set A}) be the interpolation operator. We have the following
interpolation error estimates (see \cite{SuLiRectan,SuLi2018}).

\begin{Lemma}
\label{known results of Su}
Let $\Omega_\rho = B_1(\bm{0})\setminus \cup_{k=1}^K B_{\rho_k}(\bm{x}_k)$, and
$\bm{u}(\bm{x}) \in \mathcal{A}\cap\mathcal{U}(\Upsilon)$.
Let $\mathscr{T}'$ be a circular ring layered mesh on
$\cup_{k=1}^K (B_{\delta_k}(\bm{x}_k)\setminus B_{\rho_k}(\bm{x}_k))$ satisfying that,
for a given constant $\alpha\in(0,1)$,
$\tau\lesssim \min\{\sqrt{\epsilon},\epsilon^{(1-\alpha)/4} h\}$, and
$N^{-1}\lesssim \epsilon^{(1-\alpha)/4} h$ if the layer is a standard layer, while
$N \lesssim \min\{(\epsilon\tau)^{1/4},(\tau h^2)^{1/4}\}$ if the layer is a conforming layer.
Denote $\Omega_k=B_{\delta_k}(\bm{x}_k)\setminus B_{\rho_k}(\bm{x}_k)$ for $1\le k\le K$,
and $\Omega_{K+1}=B_1(\bm{0})\setminus \cup_{k=1}^KB_{\delta_k}(\bm{x}_k)$.
Then, we have
\begin{equation}
\label{known error1}
\begin{aligned}
|\det\nabla\Pi_h^2\bm{u}(\bm{x})-\det\nabla\bm{u}(\bm{x})|\lesssim
|\bm{x}-\bm{x}_k|^{-1}(\tau^2+N^{-2}),
\quad \forall \bm{x} \in \Omega_k, \; 1\le k \le K,
\end{aligned}
\end{equation}
\begin{equation}
\label{known error1.5}
|\det\nabla\Pi_h^2\bm{u}(\bm{x})-\det\nabla\bm{u}(\bm{x})|\lesssim h^2, \quad
\forall \bm{x} \in \Omega_{K+1},
\end{equation}
\begin{equation}
\begin{aligned}
\label{known error2}
\| \det\nabla\Pi^2_h\bm{u}-\det\nabla\bm{u}\|_{0,2,\Omega_k}\lesssim h^2, \quad 1\le k \le K+1,
\end{aligned}
\end{equation}
\begin{equation}
\begin{aligned}
\label{known error3}
\Big|\int_{\Omega_k}|\Pi^2_h\bm{u}(\bm{x})|^s-
|\bm{u}(\bm{x})|^s \,\mathrm{d}\bm{x}\Big| \lesssim h^2, \quad 1\le k \le K+1,
\end{aligned}
\end{equation}
\begin{equation}
\begin{aligned}
\label{known error4}
\Big|\int_{\Omega_k}d(\det\nabla\Pi^2_h\bm{u})
-d(\det\nabla\bm{u}) \,\mathrm{d}\bm{x}\Big| \lesssim h^3 \quad 1\le k \le K+1.
\end{aligned}
\end{equation}
\end{Lemma}

\begin{Lemma}
\label{lemma energy error}
Let $(\bm{\tilde{u}},\tilde{p}) \in (\mathcal{A}\cap \mathcal{U}(\Upsilon))\times
H^{\gamma+1}(\Omega_\rho)$, $\gamma =0$ or $1$, be a solution to problem
\eqref{weak form of Euler Lagrange equation} with $\bm{\tilde{u}}$ being an absolute minimizer
of $E(\cdot)$ in $\mathcal{A}_I$.
Let $\mathscr{T}'$ satisfy the conditions in Lemma~\ref{known results of Su}.
Let $(\bm{u}_h, p_h)\in \mathcal{X}_h\times
\mathcal{M}_h$ be the finite element solutions
to problem \eqref{discrete Euler Lagrange equation} with
$\|p_h\|_{0,2,\Omega_{\rho}} \lesssim h^{-\beta}$ for a constant $\beta\in[0,2)$. Then
\begin{equation}
\label{theoretical error of FES energy}
-h^{\gamma+1}\lesssim E(\bm{u}_h, p_h)-E(\bm{\tilde{u}},\tilde{p}) \lesssim h^{2-\beta},
\end{equation}
\begin{equation}
\label{u_h det u_h bounded}
\|\bm{u}_h\|_{1,s,\Omega_{\rho}} \lesssim 1, \quad
\|\det\nabla \bm{u}_h \|_{0,2,\Omega_\rho} \lesssim 1.
\end{equation}
\end{Lemma}
\begin{proof}
Firstly, $(\bm{u}_h, p_h)$ solves problem \eqref{discrete Euler Lagrange equation} implies
$\bm{u}_h$ minimizes $E(\bm{v}_h,p_h)$ in $\mathcal{X}_h$, hence
$E(\bm{u}_h,p_h)\le E(\Pi^2_h \bm{\tilde{u}},p_h)$. On the other hand, since
$\det\nabla\tilde{\bm{u}}=1, a.e.$ in $\Omega_\rho$, one has
\begin{equation*}
\begin{aligned}
& E(\Pi^2_h \bm{\tilde{u}}, p_h)
=E(\bm{\tilde{u}},\tilde{p})+\int_{\Omega_{\rho}}\mu\big(|\nabla\Pi^2_h
\bm{\tilde{u}}|^s-|\nabla\bm{\tilde{u}}|^s\big) +\big(d(\det\nabla\Pi_h^2\bm{\tilde{u}})
-d(\det\nabla\bm{\tilde{u}})\big)\, \mathrm{d}\bm{x}\\
&\qquad \qquad \qquad \qquad \qquad \qquad \quad  -\int_{\Omega_{\rho}} p_h(\det\nabla\Pi^2_h\bm{\tilde{u}}-1)
\,\mathrm{d}\bm{x}=E(\bm{\tilde{u}},\tilde{p})+ I_1+I_2+I_3.
\end{aligned}
\end{equation*}
It follows from \eqref{known error3} and \eqref{known error4} in Lemma~\ref{known results of Su}
that $|I_1|\lesssim h^2$ and $|I_2|\lesssim h^3$. By the H\"{o}lder inequality,
\eqref{known error2} and $\|p_h\|_{0,2,\Omega_{\rho}} \lesssim h^{-\beta}$, one concludes that
\begin{equation*}
|I_3|\le\| p_h\|_{0,2,\Omega_\rho}\|\det\nabla\Pi^2_h
\bm{\tilde{u}}-1\|_{0,2,\Omega_\rho}\lesssim h^{2-\beta}.
\end{equation*}
Thus, the second relationship in
\eqref{theoretical error of FES energy} holds, and consequently $E(\bm{u}_h, p_h) \lesssim 1$,
which implies \eqref{u_h det u_h bounded}, since $d(\det\nabla \bm{u}_h)>0$ and
$\int_{\Omega_{\rho}}p_h(\det\nabla \bm{u}_h-1)\,\mathrm{d}\bm{x}=0$.

Secondly, due to $(\bm{\tilde{u}},\tilde{p}) = \arg \sup_{q \in L^2(\Omega_\rho)}
\inf_{\bm{v} \in \mathcal{A}} E(\bm{v},q)$, $\bm{\tilde{u}}$ minimizes
$E(\bm{v},\tilde{p})$ in $\mathcal{A}$. Hence, by $\int_{\Omega_{\rho}}q_h(\det\nabla \bm{u}_h-1)
\,\mathrm{d}\bm{x}=0$, $\forall q_h\in \mathcal{M}_h$ (see \eqref{discrete Euler Lagrange equation}),
one has
\begin{equation*}
E(\bm{\tilde{u}},\tilde{p})\le E(\bm{u}_h,\tilde{p})
= E(\bm{u}_h,p_h)-\int_{\Omega_\rho}(\tilde{p}-\operatorname{P}_h^{\gamma} \tilde{p})
(\det\nabla\bm{u}_h-1) \,\mathrm{d}\bm{x},
\end{equation*}
where $\operatorname{P}_h^{\gamma}:H^{\gamma+1}(\Omega_\rho)\to \mathcal{M}_h$, $\gamma=0$
or $1$, is an orthogonal projection operator with $\operatorname{P}_h^{\gamma}|_T=
\operatorname{P}_T^{\gamma}: H^{\gamma+1}(T)\to F_T(P_{\gamma})$ being defined by
\begin{equation*}
\int_{T}q (\operatorname{P}_{T}^{\gamma}p-p)\det \nabla\hat{\bm{x}}\, \mathrm{d}\bm{x} = 0,
\quad \forall q\in F_T(P_{\gamma}), \;\;\; \forall T \in \mathscr{T}.
\end{equation*}
Since $\mathscr{T}$ satisfies (H1), we have $\|\operatorname{P}_h^{\gamma}\tilde{p}
-\tilde{p}\|_{0,2,\Omega_\rho}
\lesssim h^{\gamma+1} |\tilde{p}|_{\gamma+1,2,\Omega_\rho}$ (see \cite{Wang2013}). In addition,
by the H\"{o}lder inequality,
\begin{equation*}
\Big|\int_{\Omega_\rho}(\tilde{p}-\operatorname{P}_h^{\gamma} \tilde{p})
(\det\nabla\bm{u}_h-1) \,\mathrm{d}\bm{x}\Big|
\lesssim  \|\tilde{p}-\operatorname{P}_h^{\gamma}\tilde{p}\|_{0,2,\Omega_\rho}
\|\det\nabla\bm{u}_h-1\|_{0,2,\Omega_\rho}.
\end{equation*}
Thus, the first relationship in \eqref{theoretical error of FES energy} holds.
\end{proof}
\begin{rem}
$O(h^{\gamma+1})$ in the energy error bounds \eqref{theoretical error of FES energy} is not
optimal and can be improved at least to $o(h^{\gamma+1})$, since
$\|\det\nabla\bm{u}_h-1\|_{0,2,\Omega_\rho} \to 0$
(see \eqref{convergence of fes}).
\end{rem}

\begin{Theorem}
\label{Convergence of uh} Let $(\bm{\tilde{u}},\tilde{p}) \in (\mathcal{A}\cap \mathcal{U}(\Upsilon))\times H^1(\Omega_\rho)$ be a solution to problem \eqref{weak form of Euler Lagrange equation} with
$\bm{\tilde{u}}$ being an absolute minimizer of $E(\cdot)$ in $\mathcal{A}_I$.
Let $\mathscr{T}$ and $(\bm{u}_h, p_h)\in \mathcal{X}_h\times
\mathcal{M}_h$ satisfy the same conditions as in Lemma~\ref{lemma energy error}.
Then, there exist a subsequence $\{\bm{u}_h\}_{h>0}$ (not relabeled)
and an absolute energy minimizer $\bm{\bar{u}}$ of $E(\cdot)$ in $\mathcal{A}_I$, such that
\begin{equation}
\label{convergence of fes}
\bm{u}_h\to \bm{\bar{u}}\ \ \text{in}\ W^{1,s}(\Omega_{\rho};\mathbb{R}^2),\quad
\det\nabla\bm{u}_h \to 1 \ \ \text{in}\ L^2(\Omega_{\rho}), \quad \text{as}\ \ h\to 0.
\end{equation}
Furthermore, if $\|p_h\|_{0,2,\Omega_{\rho}} \lesssim 1$, then there exist a subsequence
$\{p_h\}_{h>0}$ (not relabeled) and a function $\bar{p} \in L^2(\Omega_{\rho})$,
such that $(\bm{\bar{u}}, \bar{p})$ solves
problem \eqref{mixed formulation} and
\begin{equation}
\label{weak convergence of p_h}
p_h \rightharpoonup \bar{p} \quad \text{in} \;\; L^2(\Omega_{\rho}),\quad \text{as} \ h\to0.
\end{equation}
\end{Theorem}

\begin{proof}
Since $1<s<2$, \eqref{u_h det u_h bounded} implies that
there exist a subsequence $\{\bm{u}_h\}_{h>0}$ (not relabeled) and functions
$\bm{\bar{u}}\in W^{1,s}(\Omega_{\rho};\mathbb{R}^2)$, $\vartheta\in L^2(\Omega_\rho)$
such that
\begin{equation}\label{u_h det u_h converge}
\bm{u}_h\rightharpoonup \bm{\bar{u}} \;\text{in} \; W^{1,s}(\Omega_{\rho};\mathbb{R}^2),
\;\;\; \bm{u}_h\to \bm{\bar{u}} \;\;\text{and} \; \det\nabla\bm{u}_h\rightharpoonup
\vartheta \; \text{in} \; L^2(\Omega_\rho),  \;\;
\text{as} \;\; h\to 0.
\end{equation}
Thanks to some prominent results for the cavitation problems (see Theorem 3 in \cite{Henao 2011},
Theorem 2 and Theorem 3 in \cite{Henao 2010}) that, in our case (see also in \cite{SuLi2018}
for more general cases), \eqref{u_h det u_h converge} together with the continuity of
$\bm{u}_h$ actually lead to
\begin{equation}
\label{theta=det Du}
\vartheta = \det\nabla\bm{\bar{u}}, \;\; \text{a.e. in} \;\;  \Omega_{\rho}, \;\;\text{and}
\;\; \bm{\bar{u}} \;\; \text{is 1-to-1 a.e. in} \;\; \Omega_{\rho}.
\end{equation}
In addition, due to $\int_{\Omega_{\rho}}q_h(\det\nabla \bm{u}_h-1) \,\mathrm{d}\bm{x}=0$, $\forall q_h\in \mathcal{M}_h$, we have
\begin{equation*}
\int_{\Omega_\rho}q(1-\vartheta)\, \mathrm{d}\bm{x}=
\int_{\Omega_\rho}q(\det\nabla\bm{u}_h-\vartheta)\,\mathrm{d}\bm{x} -
\int_{\Omega_\rho}(q-\operatorname{P}_h^0 q)(\det\nabla\bm{u}_h-1)\,\mathrm{d}\bm{x},
\;\forall q\in C_0^{\infty}(\Omega_\rho).
\end{equation*}
By the H\"{o}lder inequality, it follows from
$\|\operatorname{P}_h^{0}q-q\|_{0,2,\Omega_\rho}\lesssim h
|q|_{1,2,\Omega_\rho}$ and \eqref{u_h det u_h bounded} that
\begin{equation*}
\begin{aligned}
\Big|\int_{\Omega_\rho}(q-\operatorname{P}_h^0 q)(\det\nabla\bm{u}_h-1)\,\mathrm{d}\bm{x}
\Big|\le \| q-\operatorname{P}_h^0 q\|_{0,2,\Omega_\rho}
\| \det\nabla\bm{u}_h-1\|_{0,2,\Omega_\rho}\to 0,\ \text{as}\ h\to 0.
\end{aligned}
\end{equation*}
Hence, by \eqref{u_h det u_h converge}-\eqref{theta=det Du}, we have
$\det\nabla\bm{\bar{u}}=\vartheta=1$, a.e. in $\Omega_{\rho}$. Furthermore,
since $\bm{u}_h\to \bm{\bar{u}}$ in $L^s(\partial\Omega_\rho)$ and
$\bm{u}_h|_{\partial_D\Omega_\rho}=\bm{u}_0$, we also have
$\bm{\bar{u}}|_{\partial_D\Omega_\rho}=\bm{u}_0$. Thus, recalling that
$\bm{\bar{u}}$ is 1-to-1 a.e. in $\Omega_{\rho}$ by \eqref{theta=det Du},
we conclude that $\bm{\bar{u}}\in \mathcal{A}_I$.

Next, we claim that $\bm{\bar{u}}$ is an absolute energy minimizer of $E(\cdot)$ in $\mathcal{A}_I$.
In fact, due to the convexity of both $|\nabla\bm{u}|^s$ and $d(\xi)$,                                        as a consequence of
\eqref{u_h det u_h converge} and \eqref{theta=det Du}, we obtain
\begin{eqnarray}
\label{lsc of int u}
\int_{\Omega_\rho}|\nabla\bm{\bar{u}}|^s \,\mathrm{d}\bm{x}
\le\liminf_{h\to 0}\int_{\Omega_\rho}|\nabla\bm{u}_h|^s\,\mathrm{d}\bm{x},&
\\ \label{lsc of int d}
\int_{\Omega_\rho} d(\det\nabla\bm{\bar{u}})\,\mathrm{d}\bm{x}
\le\liminf_{h\to 0}\int_{\Omega_\rho}d(\det\nabla\bm{\bm{u}_h}) \,\mathrm{d}\bm{x}.&
\end{eqnarray}
Hence, by Lemma~\ref{lemma energy error} (see \eqref{theoretical error of FES energy}), we have
\begin{equation}
\label{bar u is minimizer}
\inf_{\bm{v}\in \mathcal{A}_I} E(\bm{v})\le E(\bm{\bar{u}})\le\liminf\limits_{h\to 0}E(\bm{u}_h)=
\liminf\limits_{h\to 0}E(\bm{u}_h,p_h)=E(\tilde{\bm{u}},\tilde{p})=
\inf_{\bm{v}\in \mathcal{A}_I} E(\bm{v}).
\end{equation}

Now, we are going to show the strong convergence of $\bm{u}_h$. Notice that
\eqref{bar u is minimizer} implies $E(\bm{\bar{u}})=\lim\limits_{h\to 0}E(\bm{u}_h)$,
by \eqref{lsc of int d}, we have
\begin{equation}
\label{d <= liminf d_h}
\begin{aligned}
&E(\bm{\bar{u}}) - \int_{\Omega_\rho}\mu|\nabla\bm{\bar{u}}|^s\, \mathrm{d}\bm{x} =
 \int_{\Omega_\rho}d(\det\nabla\bm{\bar{u}})\, \mathrm{d}\bm{x}
\le  \liminf_{h\to 0} \int_{\Omega_{\rho}}d(\det\nabla\bm{u}_h)\, \mathrm{d}\bm{x}\\
&\qquad \quad \ \ =
\liminf_{h\to 0} \big(E(\bm{u}_ h) - \int_{\Omega_\rho}\mu|\nabla\bm{u}_h|^s\, \mathrm{d}\bm{x}\big)= E(\bm{\bar{u}})
 -\limsup_{h\to 0}\int_{\Omega_\rho}\mu|\nabla\bm{u}_h|^s\, \mathrm{d}\bm{x},
\end{aligned}
\end{equation}
i.e. $\limsup\limits_{h\to0}\vert\bm{u}_h\vert_{1,s,\Omega_\rho}
\le \vert\bm{\bar{u}}\vert_{1,s,\Omega_\rho}$. This together with \eqref{lsc of int u}
yields $\lim\limits_{h\to 0}\vert \bm{u}_h\vert _{1,s,\Omega_\rho}=
\vert \bm{\bar{u}}\vert _{1,s,\Omega_\rho}$. Recall $W^{1,s}(\Omega_{\rho};\mathbb{R}^2)$
enjoys the Radon-Riesz property (see \cite{Megginson1998}), $\bm{u}_h\to\bm{\bar{u}}$ in
$W^{1,s}(\Omega_{\rho};\mathbb{R}^2)$ follows as a consequence of
$\bm{u}_h\rightharpoonup \bm{\bar{u}} \;\text{in} \; W^{1,s}(\Omega_{\rho};\mathbb{R}^2)$
and $\lim\limits_{h\to 0}\vert \bm{u}_h\vert _{1,s,\Omega_\rho}=
\vert \bm{\bar{u}}\vert _{1,s,\Omega_\rho}$.

In addition, by $\bm{u}_h \to \bar{\bm{u}}$ in $W^{1,s}(\Omega_\rho;\mathbb{R}^2)$,
the inequality in \eqref{d <= liminf d_h}
is actually an equality, hence we have $\lim\limits_{h\to 0}\int_{\Omega_\rho}
d(\det\nabla\bm{u}_h)\, \mathrm{d}\bm{x}=\int_{\Omega_\rho} d(\det\nabla\bm{\bar{u}})\,
\mathrm{d}\bm{x}$. Thus, it follows from $\det\nabla\bm{u}_h\rightharpoonup
\det\nabla\bm{\bar{u}}$ in $L^2(\Omega_\rho)$ and the convexity of $d_1(\cdot)$ that
\begin{equation*}
\begin{aligned}
\int_{\Omega_\rho}d(\det\nabla\bm{\bar{u}})& -\kappa(\det\nabla\bm{\bar{u}}-1)^2\,
\mathrm{d}\bm{x} = \int_{\Omega_\rho} d_1(\det\nabla\bm{\bar{u}})\, \mathrm{d}\bm{x}
\le \liminf_{h\to 0}\int_{\Omega_\rho} d_1(\det\nabla\bm{u}_h)\, \mathrm{d}\bm{x} \\
&\qquad \qquad \qquad \quad =\liminf_{h\to 0}
\int_{\Omega_\rho}d(\det\nabla\bm{u}_h)-\kappa(\det\nabla\bm{u}_h-1)^2\,
\mathrm{d}\bm{x} \\
&\qquad \qquad \qquad \quad =\int_{\Omega_\rho}d(\det\nabla\bm{\bar{u}})\, \mathrm{d}\bm{x}-
\limsup_{h\to 0}\int_{\Omega_\rho}\kappa(\det\nabla\bm{u}_h-1)^2\, \mathrm{d}\bm{x},
\end{aligned}
\end{equation*}
i.e., $\limsup\limits_{h\to 0}\int_{\Omega_\rho}(\det\nabla\bm{u}_h-1)^2\,
\mathrm{d}\bm{x}\le\int_{\Omega_\rho}(\det\nabla\bm{\bar{u}}-1)^2\, \mathrm{d}\bm{x}=0$,
which means $\det\nabla\bm{u}_h \to 1$ in $L^2(\Omega_\rho)$ as $h\to 0$.

Finally, $\|p_h\|_{0,2,\Omega_{\rho}} \lesssim 1$ implies that there exist a subsequence
$\{p_h\}_{h>0}$ (not relabeled) and a function $\bar{p} \in L^2(\Omega_{\rho})$,
such that \eqref{weak convergence of p_h} holds. Thus, by $\det\nabla\bm{u}_h \to
\det \nabla \bm{\bar{u}}=1$ in $L^2(\Omega_\rho)$ and \eqref{bar u is minimizer},
we have $E(\bm{\bar{u}},\bar{p})=\lim\limits_{h\to 0}E(\bm{u}_h,p_h)=E(\bm{\tilde{u}},\tilde{p})$,
which completes the proof of the theorem.
\end{proof}

\begin{Theorem}
\label{Convergence of ph} Let $(\bm{\tilde{u}},\tilde{p})
\in (\mathcal{A}\cap \mathcal{U}(\Upsilon))\times H^1(\Omega_\rho)$ be a solution to problem \eqref{weak form of Euler Lagrange equation} with
$\bm{\tilde{u}}$ being an absolute minimizer of $E(\cdot)$ in $\mathcal{A}_I$.
Let $\mathscr{T}$ and $(\bm{u}_h, p_h)\in \mathcal{X}_h\times
\mathcal{M}_h$ satisfy the same conditions as in Lemma~\ref{lemma energy error}.
Let $(\bm{\bar{u}},\bar{p})$ be given by Theorem~\ref{Convergence of uh}.
If, in addition, $\bar{p}\in H^1(\Omega_{\rho})$, $\|p_h\|_{0,2,\Omega_{\rho}} \lesssim 1$,
$\|\bm{u}_h\|_{1,\zeta,\Omega_\rho} \lesssim 1$ and $c\le\det \nabla \bm{u}_h \le C$,
a.e. in $\Omega_{\rho}$, where $\zeta>2$ and $0<c<1<C$ are constants independent of $h$.
Then
\begin{equation}
\label{strong convergence of ph}
p_h\rightarrow \bar{p}\;\; \text{in}\;\; L^2(\Omega_\rho),\quad \text{as}\quad h\to 0.
\end{equation}
\end{Theorem}
\begin{proof}
Firstly, we see that $\|\bm{u}_h\|_{1,\zeta,\Omega_\rho} \lesssim 1$
implies $\bm{\bar{u}}\in W^{1,\zeta}(\Omega_{\rho};\mathbb{R}^2)$, and thus it follows from (see the interpolation inequality on page 125 in \cite{Nirenberg1959})
\begin{equation}
\begin{aligned}
\|\nabla\bm{u}_h-\nabla\bar{\bm{u}}\|_{0,2,\Omega_\rho}\le \|\nabla\bm{u}_h-\nabla\bar{\bm{u}}\|_{0,\zeta,\Omega_\rho}^{1-\alpha}
\|\nabla\bm{u}_h-\nabla\bar{\bm{u}}\|_{0,s,\Omega_\rho}^{\alpha},
\end{aligned}
\end{equation}
where $0<\alpha<1$ is determined by $\frac{1}{2}=\frac{1-\alpha}{\zeta}+\frac{\alpha}{s}$,
that $\bm{u}_h\to \bm{\bar{u}}$ in $W^{1,s}(\Omega_{\rho};\mathbb{R}^2)$ in
\eqref{convergence of fes} can be
strengthened to $\bm{u}_h\to \bm{\bar{u}}$ in $H^1(\Omega_{\rho})$.

We will frequently use below the facts that $|A:B|\le |A||B|$, $\forall A, B \in M^{n \times n}$,
and $|\nabla\bm{u}_h| \ge \det \nabla \bm{u}_h \ge c$
a.e. in $\Omega_{\rho}$, $|\nabla\bar{\bm{u}}| \ge \det \nabla \bm{u} = 1 >c$ a.e. in $\Omega_{\rho}$.

Secondly, we are going to show that
\begin{equation}
\label{limit2}
\lim_{h\to 0} \sup_{\bm{v}_h\in\mathcal{X}_h}
\frac{\int_{\Omega_\rho}\mu s \big(|\nabla\bar{\bm{u}}|^{s-2}\nabla\bar{\bm{u}}
- |\nabla\bm{u}_h|^{s-2}\nabla\bm{u}_h\big) :\nabla\bm{v}_h\, \mathrm{d}\bm{x}}
{|\bm{v}_h|_{1,2,\Omega_\rho}}=0,
\end{equation}
\begin{equation}
\label{limit3}
\lim_{h\to 0}\sup_{\bm{v}_h\in\mathcal{X}_h}\frac{
\int_{\Omega_\rho}\big(d'(\det\nabla\bar{\bm{u}})\operatorname{cof}\nabla\bar{\bm{u}}
- d'(\det\nabla\bm{u}_h)\operatorname{cof}\nabla\bm{u}_h\big) :\nabla\bm{v}_h\, \mathrm{d}\bm{x}}
{|\bm{v}_h|_{1,2,\Omega_\rho}}=0.
\end{equation}

In fact, as $\bm{u}_h\to \bm{\bar{u}}$ in $H^1(\Omega_{\rho};\mathbb{R}^2)$,
by the H\"{o}lder inequality, \eqref{limit2} follows from
\begin{equation*}
\Big|\int_{\Omega_\rho}|\nabla\bar{\bm{u}}|^{s-2}
(\nabla\bar{\bm{u}}-\nabla\bm{u}_h):\nabla\bm{v}_h\,\mathrm{d}\bm{x}\Big|
\le
c^{s-2}\|\nabla\bar{\bm{u}}-\nabla\bm{u}_h\|_{0,2,\Omega_\rho}
\|\nabla\bm{v}_h\|_{0,2,\Omega_\rho},
\end{equation*}
and
\begin{equation*}
\begin{aligned}
&  \Big|\int_{\Omega_\rho}(|\nabla\bar{\bm{u}}|^{s-2}-|\nabla\bm{u}_h|^{s-2})
   (\nabla\bm{u}_h:\nabla\bm{v}_h)\,\mathrm{d}\bm{x}\Big|\\
= & \Big|\int_{\Omega_\rho}(s-2)|\nabla\bm{u}_\eta|^{s-4}\big(\nabla\bm{u}_\eta:
(\nabla\bar{\bm{u}} -\nabla\bm{u}_h)\big)(\nabla\bm{u}_h:\nabla\bm{v}_h)\,\mathrm{d}\bm{x}\Big|\\
\le & c^{s-4}(s-2)\|\nabla\bar{\bm{u}}-\nabla\bm{u}_h\|_{0,2,\Omega_\rho}
    \|\nabla\bm{v}_h\|_{0,2,\Omega_\rho},
\end{aligned}
\end{equation*}
where $\nabla\bm{u}_\eta:=\nabla\bar{\bm{u}}+\eta(\nabla\bm{u}_h-\nabla\bar{\bm{u}})$
with $\eta\in(0,1)$.

Similarly, as $\bm{u}_h\to \bm{\bar{u}}$ in $H^1(\Omega_{\rho})$ and
$\det\nabla\bm{u}_h \to\det \nabla \bm{\bar{u}}=1$ in $L^2(\Omega_\rho)$,
by the H\"{o}lder inequality,
\eqref{limit3} follows as a consequence of
\begin{equation*}
\Big|\int_{\Omega_\rho}d'(\det\nabla\bar{\bm{u}}) (\operatorname{cof}
\nabla\bar{\bm{u}}- \operatorname{cof}\nabla\bm{u}_h):\nabla\bm{v}_h\,
\mathrm{d}\bm{x}\Big|
\le |d'(1)|
\|\nabla\bm{u}_h-\nabla\bar{\bm{u}}\|_{0,2,\Omega_\rho}
\|\nabla\bm{v}_h\|_{0,2,\Omega_\rho},
\end{equation*}
and
\begin{equation*}
\begin{aligned}
&\Big|\int_{\Omega_\rho}(d'(\det\nabla\bar{\bm{u}})-d'(\det\nabla\bm{u}_h))
\operatorname{cof}\nabla\bm{u}_h:\nabla\bm{v}_h\,\mathrm{d}\bm{x}\Big| \\
&\qquad \qquad \qquad \qquad \le \|d''(\eta)\|_{0,\infty,\Omega_\rho} \|\det\nabla\bar{\bm{u}}-
\det\nabla\bm{u}_h \|_{0,2,\Omega_\rho}\|\nabla\bm{u}_h\|_{0,2,\Omega_\rho}
\|\nabla\bm{v}_h\|_{0,2,\Omega_\rho},
\end{aligned}
\end{equation*}
where $\eta$ is between $\det\nabla\bm{u}_h$ and $\det\nabla\bar{\bm{u}}$, hence
$\|d''(\eta)\|_{0,\infty,\Omega_\rho}  \le
\displaystyle\max_{c\le \xi \le C} d''(\xi) |\Omega_\rho|$.

Next, we claim that
\begin{equation}
\label{P0hp-p_h->0}
\lim_{h\to 0}\sup_{\bm{v}_h\in\mathcal{X}_h}\frac{
\int_{\Omega_\rho}(\operatorname{P}_h^0\bar{p}-p_h)
\operatorname{cof}\nabla\bm{u}_h:\nabla\bm{v}_h\,\mathrm{d}\bm{x}}
{|\bm{v}_h|_{1,2,\Omega_\rho}}= 0,
\end{equation}
where $\operatorname{P}_h^0:H^1(\Omega_{\rho}) \to \mathcal{M}_h$ is defined in the proof of Lemma~\ref{lemma energy error}. In fact,
\begin{multline*}
\label{limit44}
\int_{\Omega_\rho}(\operatorname{P}_h^0\bar{p}-p_h)\operatorname{cof}\nabla\bm{u}_h:
\nabla\bm{v}_h\, \mathrm{d}\bm{x} =
\int_{\Omega_\rho}\big(\bar{p}\operatorname{cof}\nabla\bar{\bm{u}} -
p_h\operatorname{cof}\nabla\bm{u}_h\big):\nabla\bm{v}_h\,\mathrm{d}\bm{x} \\
-\int_{\Omega_\rho}\bar{p}(\operatorname{cof}\nabla\bar{\bm{u}} -
\operatorname{cof}\nabla\bm{u}_h):\nabla\bm{v}_h \, \mathrm{d}\bm{x} -
\int_{\Omega_\rho}(\bar{p}-\operatorname{P}_h^0\bar{p})
\operatorname{cof}\nabla\bm{u}_h:\nabla\bm{v}_h\,\mathrm{d}\bm{x} = I_1+I_2+I_3.
\end{multline*}
By \eqref{weak form of Euler Lagrange equation}, \eqref{discrete Euler Lagrange equation},
\eqref{limit2} and \eqref{limit3}, we have
$\lim_{h\to 0}\sup_{\bm{v}_h\in\mathcal{X}_h}\frac{|I_1|}{|\bm{v}_h|_{1,2,\Omega_\rho}} =0$. In addition,
\begin{equation*}
|I_2| \lesssim \|\bar{p}\|_{0,\infty,\Omega_\rho}
|\bar{\bm{u}}-\bm{u}_h|_{1,2,\Omega_\rho}|\bm{v}_h|_{1,2,\Omega_\rho},
\quad \forall \bm{v}_h\in\mathcal{X}_h,
\end{equation*}
\begin{equation*}
|I_3|\lesssim  \|\bar{p}-\operatorname{P}_h^0\bar{p}\|_{0,\eta,\Omega_\rho}
|\bm{u}_h|_{1,\zeta,\Omega_\rho}|\bm{v}_h|_{1,2,\Omega_\rho},
\quad \forall \bm{v}_h\in\mathcal{X}_h,
\end{equation*}
where $\eta=2\zeta/(\zeta-2)$. Since $\bar{p}\in H^1(\Omega_{\rho})$ implies that $\lim_{h\to 0}
\|\bar{p}-\operatorname{P}_h^0\bar{p}\|_{0,\xi,\Omega_\rho} =0$ and
$\|\bar{p}\|_{0,\infty,\Omega_\rho}<\infty$, we are led to $\lim_{h\to 0}
\sup_{\bm{v}_h\in\mathcal{X}_h}\frac{|I_2|+|I_3|}{|\bm{v}_h|_{1,2,\Omega_\rho}} =0$.

Finally, by the interpolation error estimate of $\operatorname{P}_h^0$ and
the LBB condition \eqref{LBB condition},
\begin{equation}
\label{p-p_h<=LBB}
\begin{aligned}
\|\bar{p}-p_h\|_{0,2,\Omega_\rho}\le& \|\bar{p}-\operatorname{P}_h^0\bar{p}\|_{0,2,\Omega_\rho}
+\|p_h-\operatorname{P}_h^0\bar{p}\|_{0,2,\Omega_\rho}\\
\lesssim & h|\bar{p}|_{1,2,\Omega_\rho}+ \sup_{\bm{v}_h\in\mathcal{X}_h}\frac{\int_{\Omega_\rho}
(p_h-\operatorname{P}_h^0\bar{p})\operatorname{cof}\nabla\bm{u}_h:\nabla\bm{v}_h
\,\mathrm{d}\bm{x}} {\|\bm{v}_h\|_{1,2,\Omega_\rho}}.
\end{aligned}
\end{equation}
Since $\bm{v}\in H_E^1(\Omega_\rho)$ and $\partial_D\Omega_\rho$ is not empty, we have
$|\bm{v}_h|_{1,2,\Omega_\rho} \cong \|\bm{v}_h\|_{1,2,\Omega_\rho}$ by
the Poincar\'{e}-Friedrichs inequality. Thus, \eqref{strong convergence of ph} follows as a consequence of
\eqref{P0hp-p_h->0} and \eqref{p-p_h<=LBB}.
\end{proof}

\section{Numerical experiments and results}

In this section, we present numerical results obtained by our method in solving the
incompressible nonlinear elasticity multi-cavity problem. In our numerical experiments,
the energy density is given by \eqref{energy density}
with $s=3/2$, $\mu=2/3$ and $d(\xi)=(\xi-1)^2/2+1/\xi$,
and the meshes $\mathscr{T}_h$ satisfy (H1), and in particular,
for properly given constants $C\ge (2-s)2^{s-1}$, $C_1$, $C_2>0$, and
$h\le \min\{\frac{2-s}{2^{2-s}C}, \frac{2-s}{2^{s-1}C} \}$, $\mathscr{T}'_h$ satisfy
additional requirements:
(1) Orientation preservation condition $\tau \le C_1\epsilon^{1/2}$, and
$N \ge C_2\epsilon^{-1/2}$ on the standard layer (see \cite{SuLiRectan}),
$N \ge C_2(\epsilon\tau)^{-1/4}$ on the conforming layer (see \cite{SuLi2018});
(2) Stability condition, which requires that (H2) as well as (H1) hold in each of the
Newton iterations (see \S~3);
(3) Quasi-optimal convergence rates condition,
which requires that, for a given constant $\alpha\in (0,1)$,
$N^{-1}\lesssim \epsilon^{(1-\alpha)/4} h$ and $\tau \lesssim \epsilon^{(1-\alpha)/4} h$
on the circular ring layers $B_{\epsilon+\tau}(\bm{x}_k)\setminus B_{\epsilon}(\bm{x}_k)$
(see Lemma~\ref{known results of Su}); (4) Sub-equi-distribution of the relative error
on the elastic energy, which requires that $(\epsilon+\tau)^{2-s} \le \epsilon^{2-s}+Ch$
(see \cite{SuLiRectan}).

\subsection{Single pre-existing defect case}

To demonstrate the numerical performance of our method and the meshing strategy on $\mathscr{T}'_h$,
we apply them to a typical single pre-existing defect cavitation problem.

Take $\Omega_{\rho}=B_1(\bm{0}) \setminus B_{\rho}(\bm{0})$ as the reference configuration,
with $\rho = 0.01$ and $0.0001$ respectively, and set $\delta = 1$. Let
$\partial_D\Omega_\rho = \partial B_1(\bm{0})$, and consider a non-radially-symmetric
Dirichlet boundary condition
$\bm{u}_0(\bm{x}) = (2.5\bm{x},2.0\bm{x})$, $\forall \bm{x}\in \partial_D\Omega_\rho$.

\begin{table}[H]{\footnotesize
\centering \subtable[$\rho=0.01$]{
\begin{tabular}{|c|c|c|c|c|}
\hline
$h$     & $\min\tau_T$  &  $\max\tau_T$  &  layers  & $N$    \\      \hline
0.06    & 0.0384        &  0.1824        &   8      &  14    \\      \hline
0.04    & 0.0224        &  0.1376        &   11     &  26    \\      \hline
0.03    & 0.0156        &  0.1164        &   14     &  34    \\      \hline
0.02    & 0.0096        &  0.0736        &   22     &  50    \\      \hline
\end{tabular}
}
\qquad
\subtable[$\rho=0.0001$]{
\begin{tabular}{|c|c|c|c|c|c|}
\hline
$h$     & $\min\tau_T$  &  $\max\tau_T$  & layers  & $\min N$  & $\max N$  \\      \hline
0.06    & 0.009         &  0.1813        &    9    &  11       &   44       \\       \hline
0.04    & 0.0080        &  0.1360        &   13    &  16       &   64        \\      \hline
0.03    & 0.0048        &  0.1056        &   17    &  21       &   84        \\      \hline
0.02    & 0.0024        &  0.0728        &   25    &  32       &   128        \\     \hline
\end{tabular}
}\caption{Data of typical meshes produced by the meshing strategy for $\mathscr{T}'_h$.}\label{table1}
}
\end{table}
Table~\ref{table1} shows two typical meshes produced by the meshing strategy for $\mathscr{T}'_h$ with
$C=2$, $C_1=2.0$, $C_2 = 2.0$, where the condition (3) does not actively take
effect on the mesh parameters. It is clearly seen that, for $\rho=0.0001$ there are two conforming
layers in the mesh, while for $\rho=0.01$ there is none.

\begin{figure}[H]
\centering \subfigure{
\label{Energy vs h 0.01}
    \includegraphics[width=2.7in, height=2in]{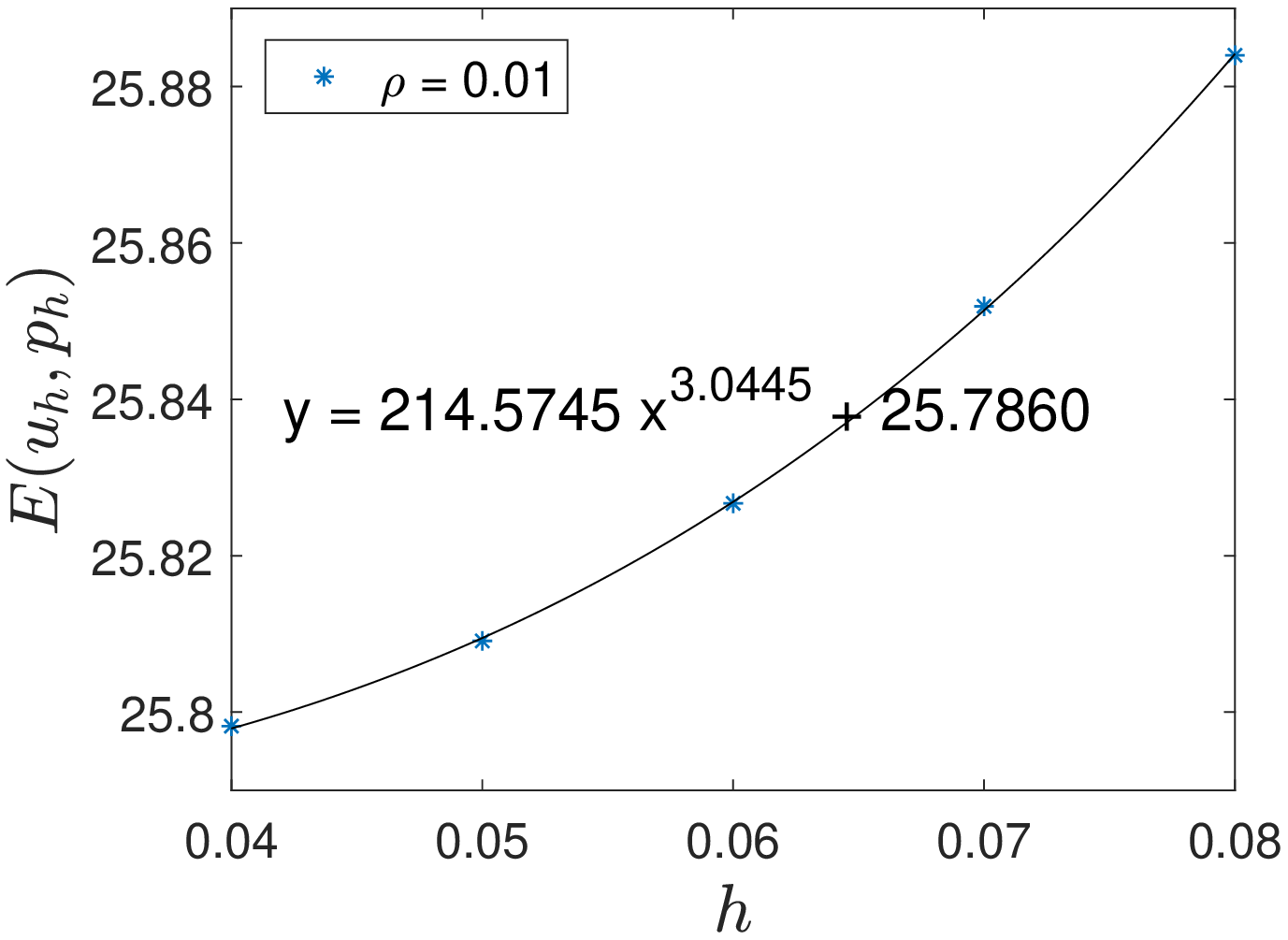}
} \hspace{1mm}
\centering \subfigure{
\label{Energy vs h 0.0001}
    \includegraphics[width=2.7in, height=2in]{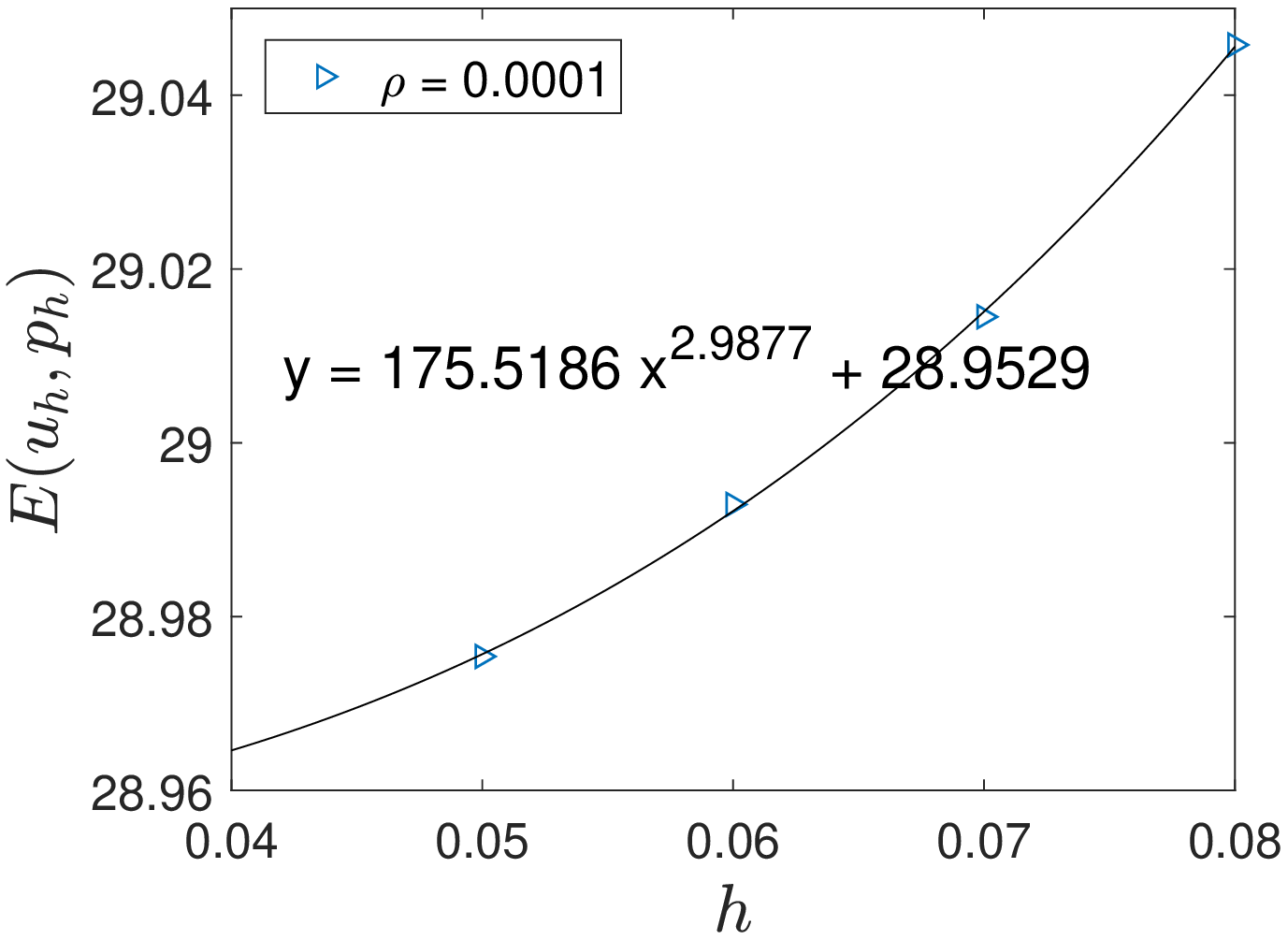}
} \vspace*{-5mm}
  \caption{Convergence behavior of the energy for $\rho=0.01$ and $0.0001$.}
\label{Energy error}
\end{figure}

\begin{figure}[H]
\centering \subfigure{
\label{uh W1s error nonradial}
    \includegraphics[width=2.75in, height=1.85in]{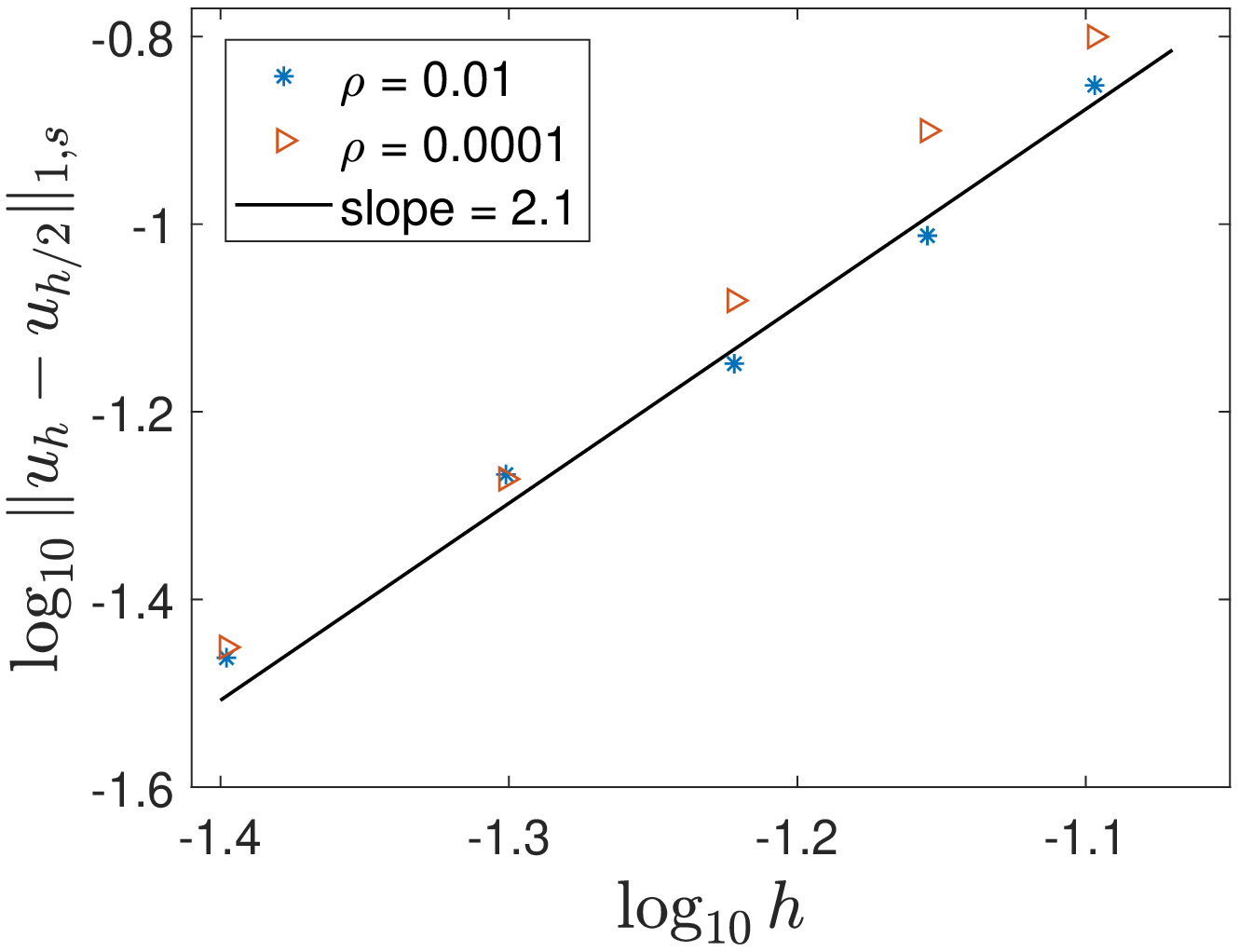}
} \hspace{1mm}
\centering \subfigure{
\label{ph L2 error nonradial}
    \includegraphics[width=2.7in, height=1.85in]{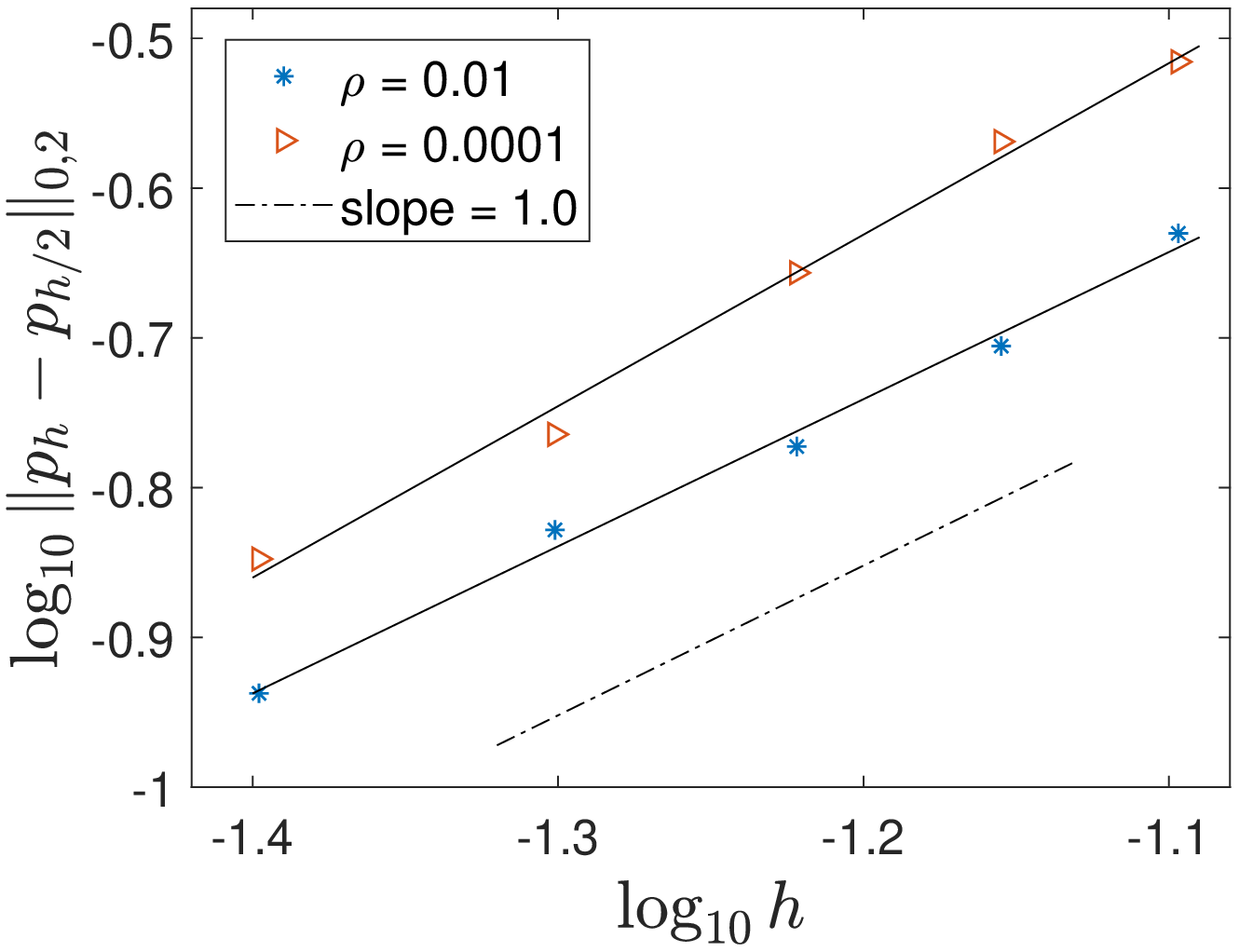}
} \vspace*{-5mm}
  \caption{Convergence behavior of $\nabla \bm{u}_h$ and $p_h$.}
\label{W1p error and Pressure error}
\end{figure}

\begin{figure}[H]
\centering \subfigure[$L^2$ error of $\det \nabla \bm{u}_h$.]{ \label{determinant L2 error nonradial}
    \includegraphics[width=2.7in, height=1.85in]{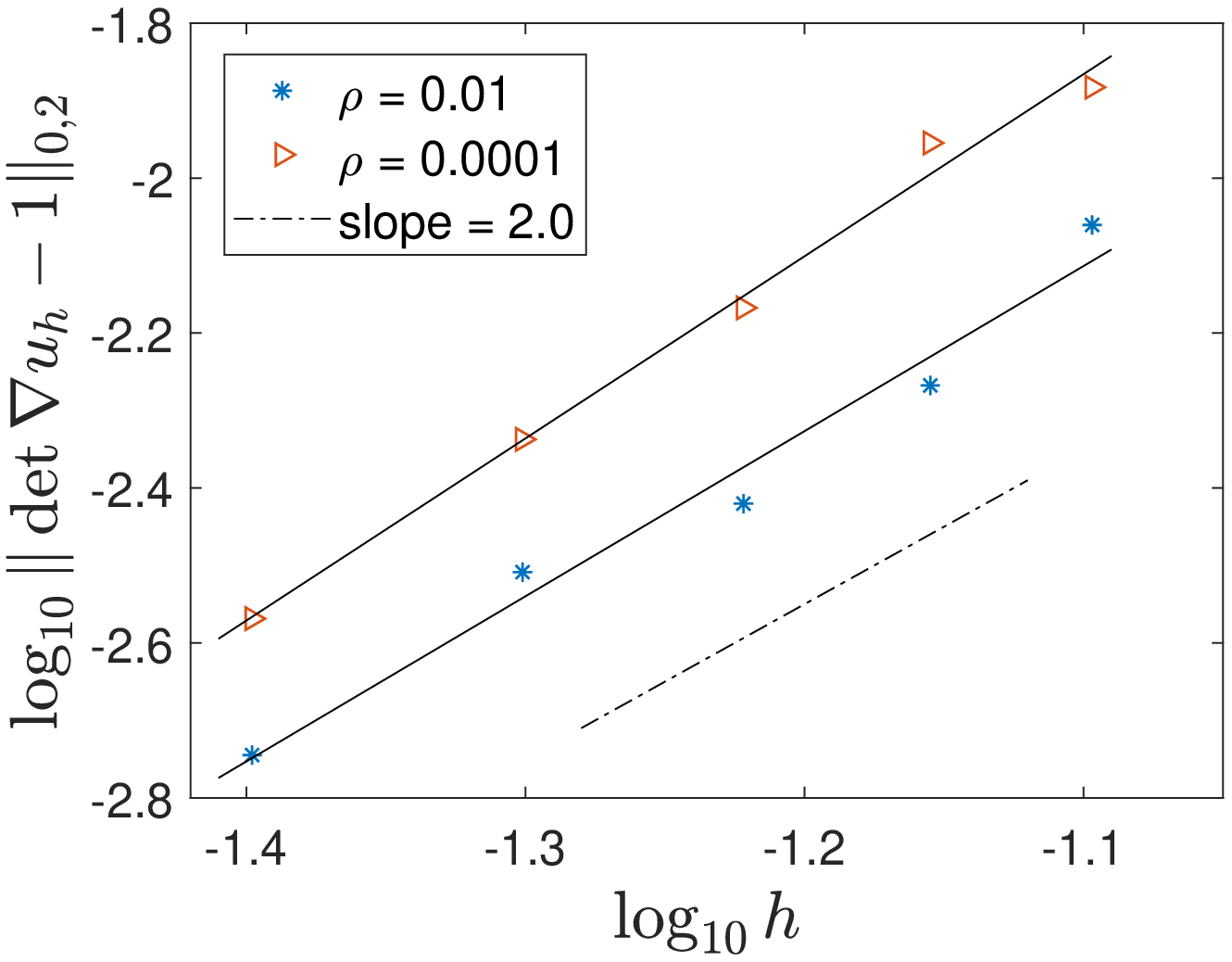}
} \hspace{1mm}
\centering \subfigure[$L^1$ error of $\det \nabla \bm{u}_h$.]{
\label{determinant L1 error nonradial}
    \includegraphics[width=2.7in, height=1.85in]{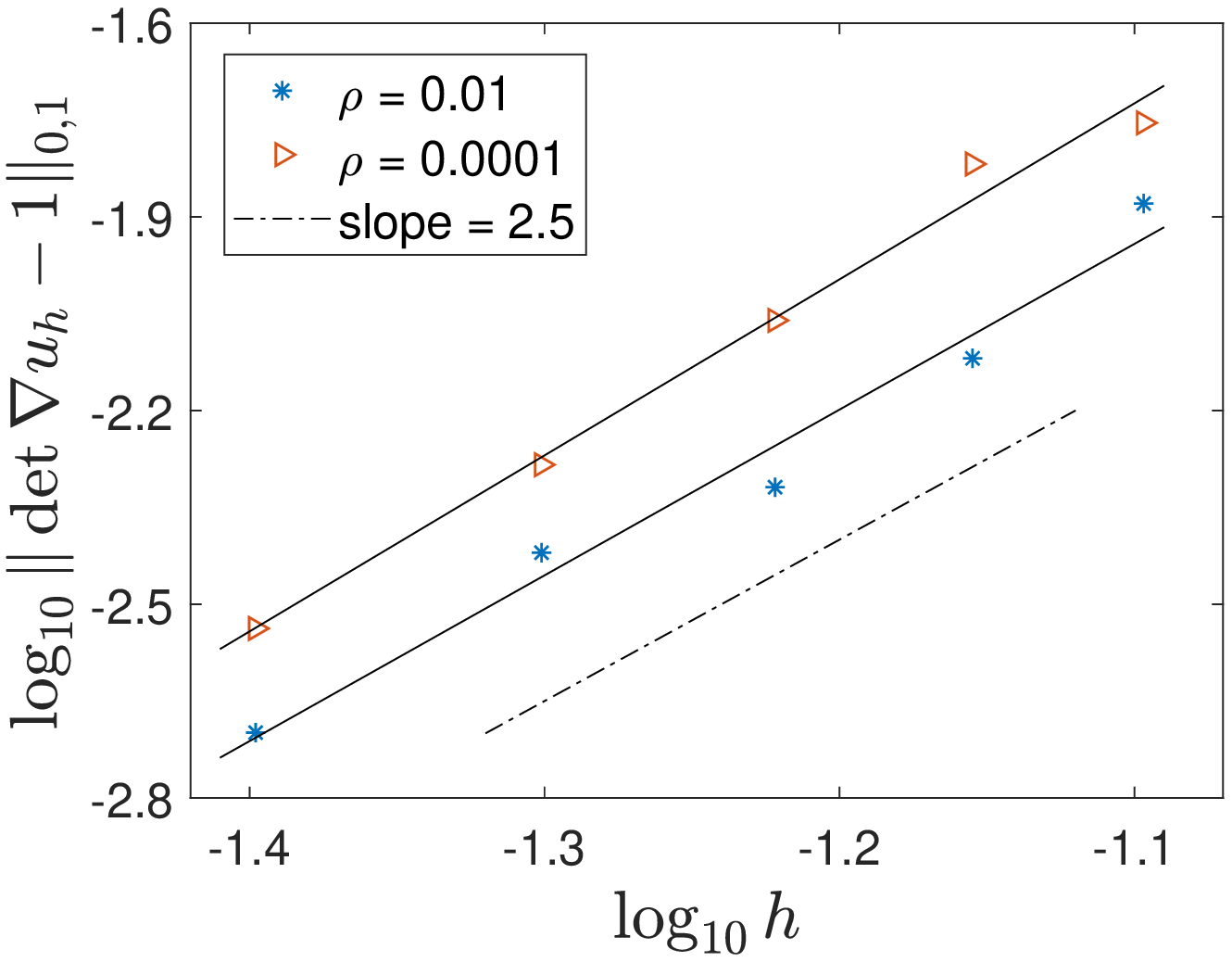}
}  \vspace*{-2mm}
\caption{Convergence behavior of $\det \nabla \bm{u}_h$.}
\label{convergence det}
\end{figure}

The convergence behavior of the numerical cavitation solutions is illustrated in
Fig~\ref{Energy error}-Fig~\ref{convergence det}, which clearly shown that the optimal
convergence rates are obtained.

\subsection{Two pre-existing defects case}

In this subsection, we consider a multi-cavity problem with 2 pre-existing defects.
Take $\Omega_{\rho}=B_1(\bm{0}) \setminus (B_{\rho_1}(\bm{x}_1)\cup B_{\rho_2}(\bm{x}_2))$
as the reference configuration, with $\rho_1=\rho_2 = 0.01$, $\bm{x_1} = (-0.2,0)$,
$\bm{x}_2 = (0.2,0)$, $\delta_1=\delta_2 = 0.15$. Let
$\partial_D\Omega_\rho = \partial B_1(\bm{0})$, and consider a Dirichlet boundary condition
$\bm{u}_0(\bm{x}) := \lambda \bm{x}$, $\forall \bm{x}\in \partial_D\Omega_\rho$,
with $1 < \lambda \le 1.4$.

Starting from an initial small deformation which is close to axisymmetric, then we are typically
led to an axisymmetric cavity solution as shown in Figure~\ref{middle}. On the other hand,
if we start from an initial deformation which is not close to axisymmetric, then we are
led to non-axisymmetric cavity solutions, in which either the left or the right cavity prevails
(see Figure~\ref{left} and Figure~\ref{right}). It is worth mentioning here that our
numerical solutions spontaneously realize the two energetically possible scenarios characterized
by Henao \& Serfaty in \cite{Henao 2013}, where it is proved that, in incompressible
nonlinear elasticity ($s=n$) multi-cavity problems, when the distance between the initial defects
are small compared with the radius of the grown cavity, either 1) cavities are pushed
together to form one equivalent round cavity (as $\rho \rightarrow 0$); or 2) all but
one cavity are of very small volume; while when the distance is much bigger, there is a third
scenario: 3) the cavities prefer to be spherical in shape and well separated.

\begin{figure}[H]
\centering \subfigure[Axisymmetric cavitation.]{
     \label{middle}
    \includegraphics[width=2.2in, height=1.54in]{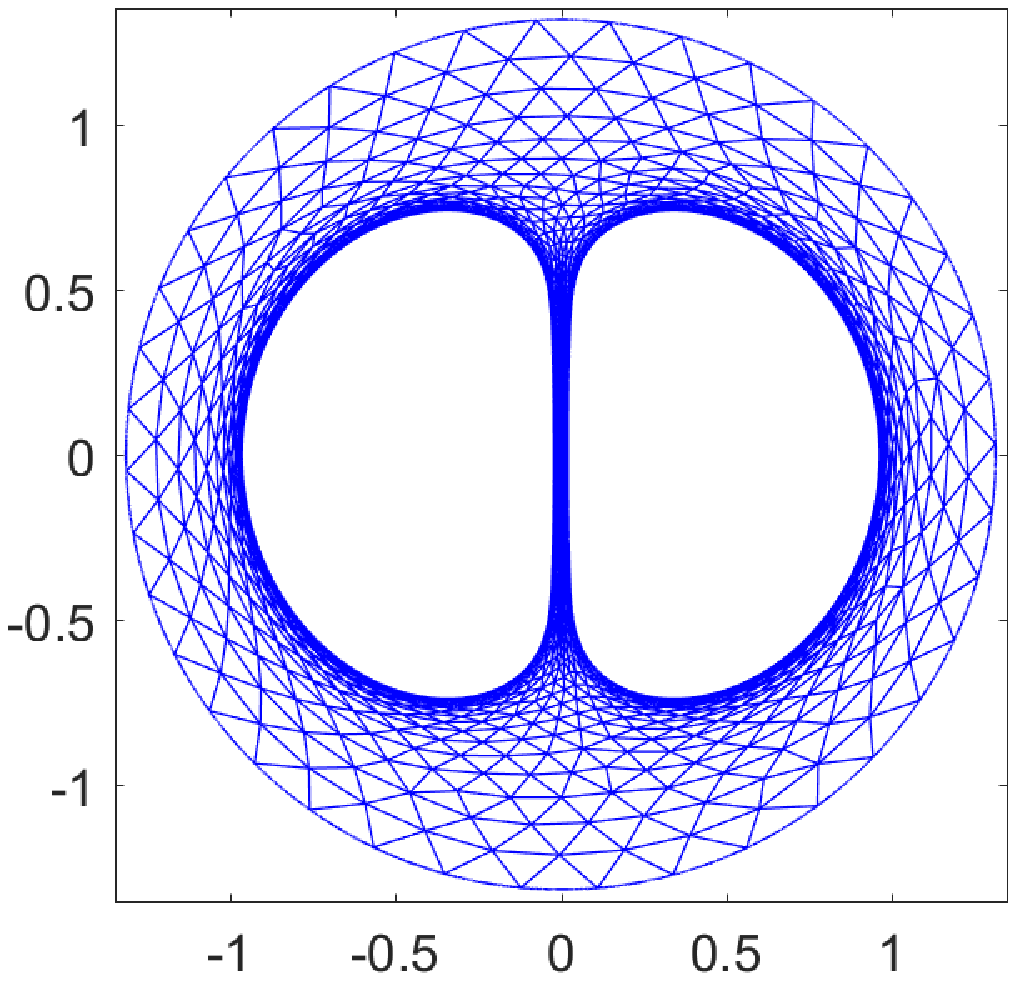}
} \hspace{-12mm}
\centering \subfigure[Left dominant.]{
\label{left}
    \includegraphics[width=2.2in, height=1.54in]{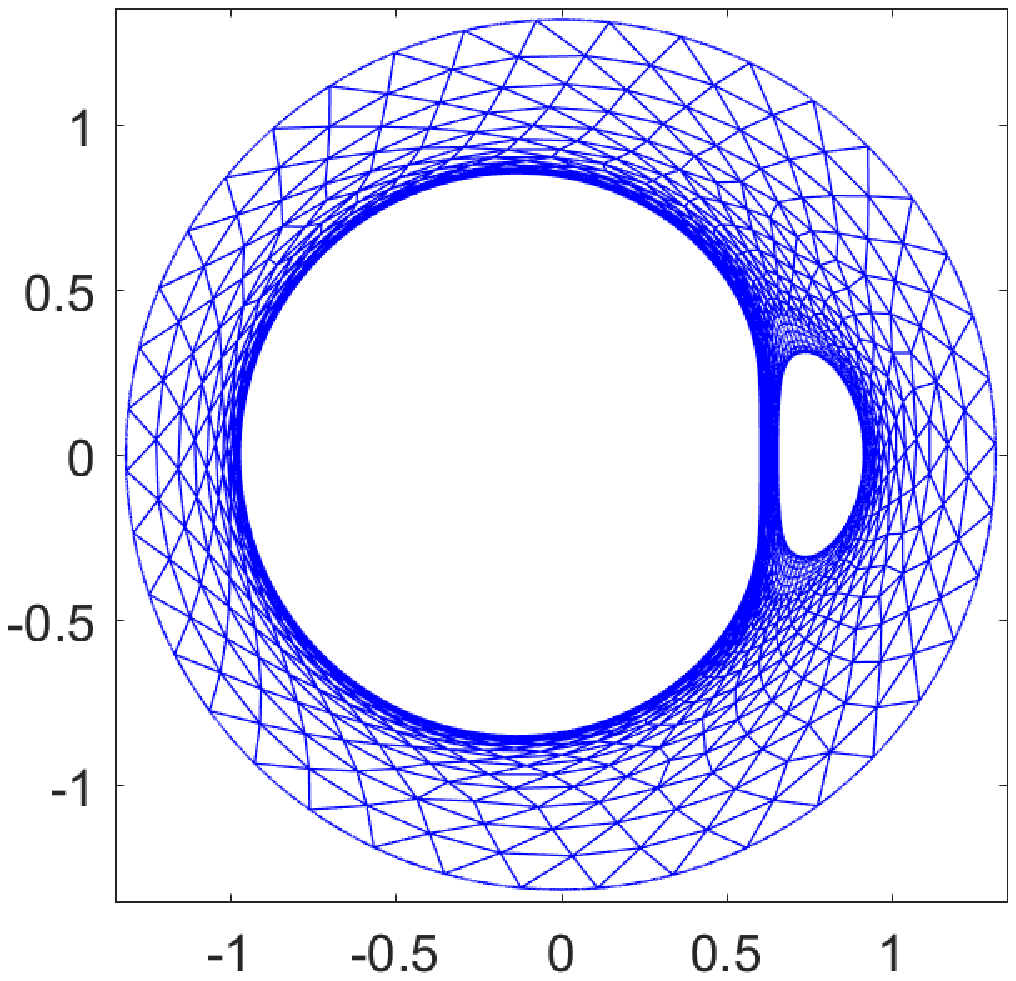}
} \hspace{-12mm}
\centering \subfigure[Right dominant.]{
\label{right}
    \includegraphics[width=2.2in, height=1.54in]{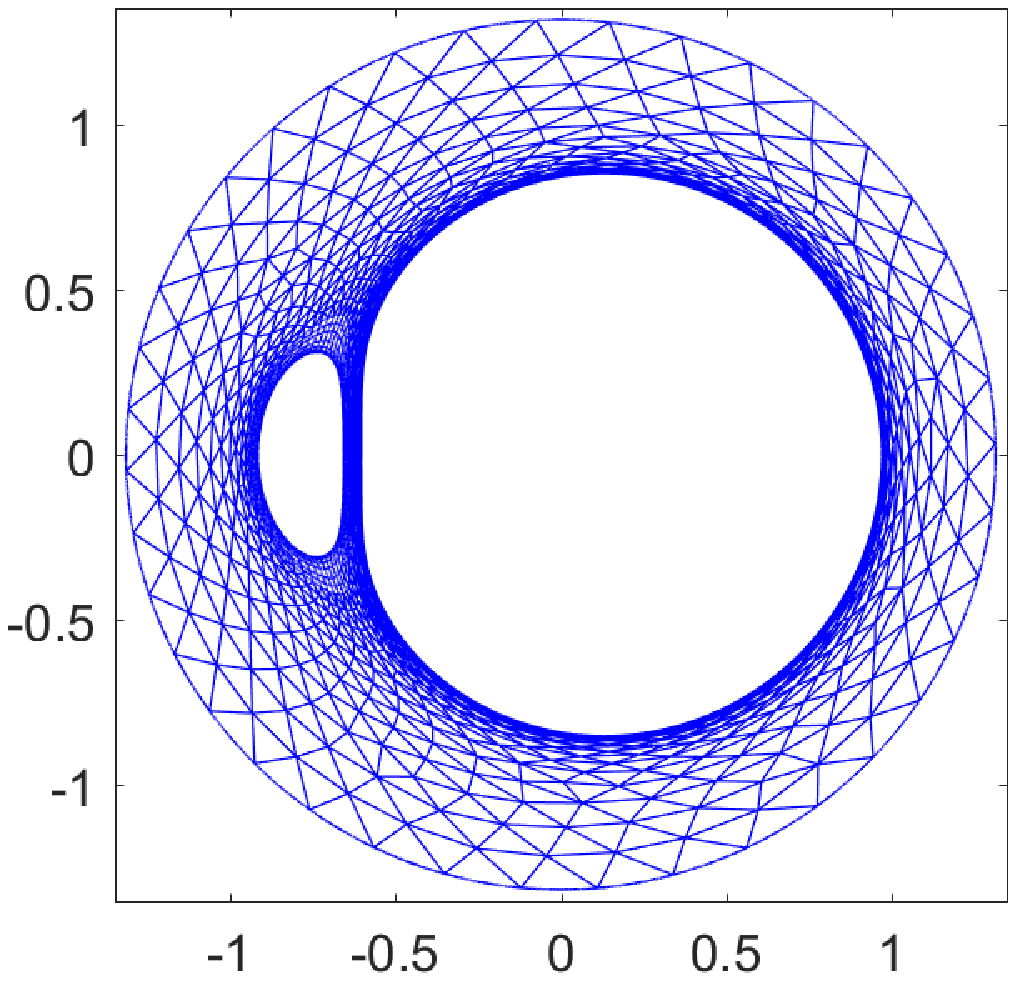}
}
  \caption{Three cavitation solutions corresponding to $\lambda=1.3$.}
\end{figure}

Our numerical experiments as well as the analytical results of Henao \& Serfaty
\cite{Henao 2013} suggest that, for a given multi-cavity problem, there should exist a
critical $\lambda_c>1$ such that, as $\lambda$ decreases across $\lambda_c$,
the multi-cavity solutions will experience a change from the scenarios 1 or 2 or both
to the scenario 3. Our numerical experiments actually captured the process of the change,
of which some snapshots are shown in Figure~\ref{middle evolution} and
Figure~\ref{right prevails evolution}, where it is obviously seen that, for $\lambda \gg 1.18$
we have cavity solutions of both scenarios 1 and 2 (see Figures~\ref{middle140} and
\ref{right140}), the difference of the two solutions narrows as
$\lambda$ deceases across $1.18$ (see Figures~\ref{middle118} and \ref{right118}),
and eventually become an indistinctive scenario 3 solution
(compare Figures~\ref{middle105} and \ref{right105}).

\begin{figure}[H]
\centering \subfigure[$\lambda = 1.4$.]{
     \label{middle140}
    \includegraphics[width=2.2in, height=1.54in]{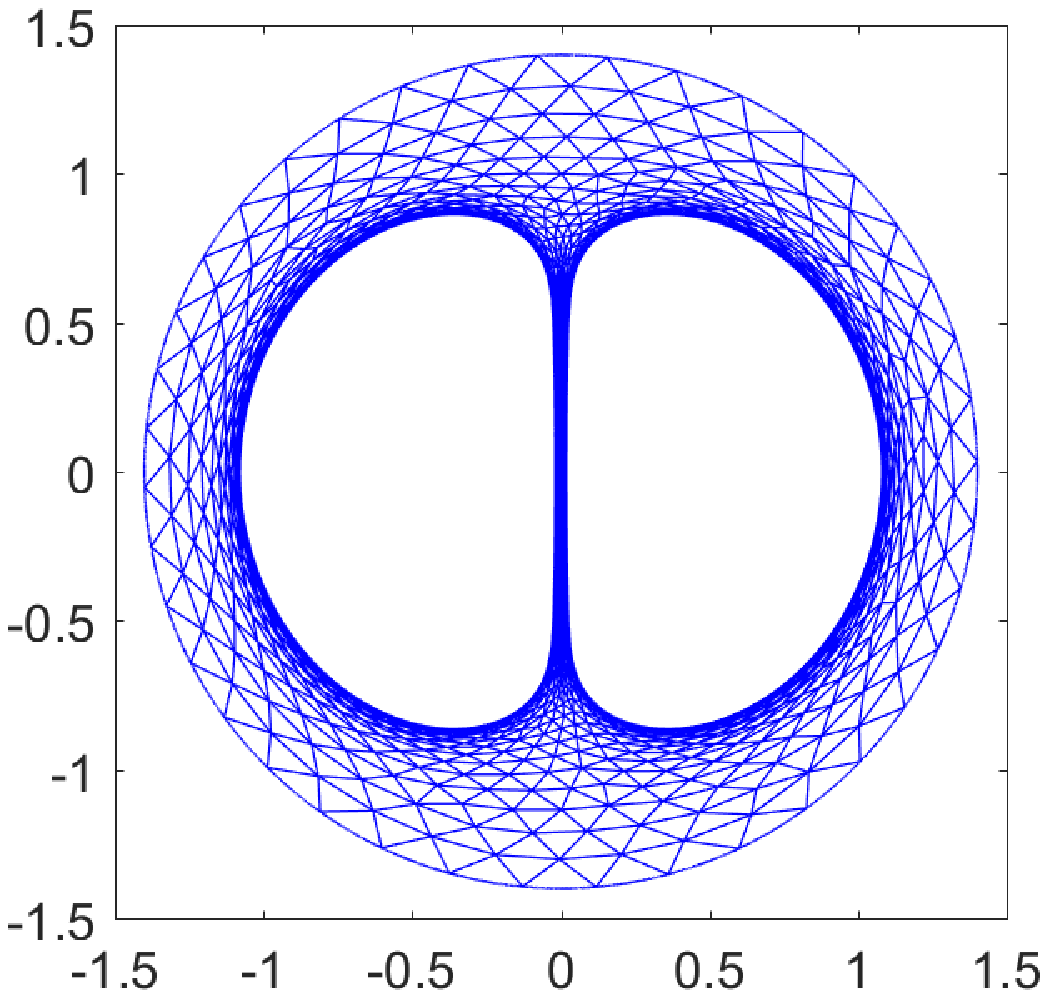}
} \hspace{-12mm}
\centering \subfigure[$\lambda = 1.18$.]{
\label{middle118}
    \includegraphics[width=2.2in, height=1.54in]{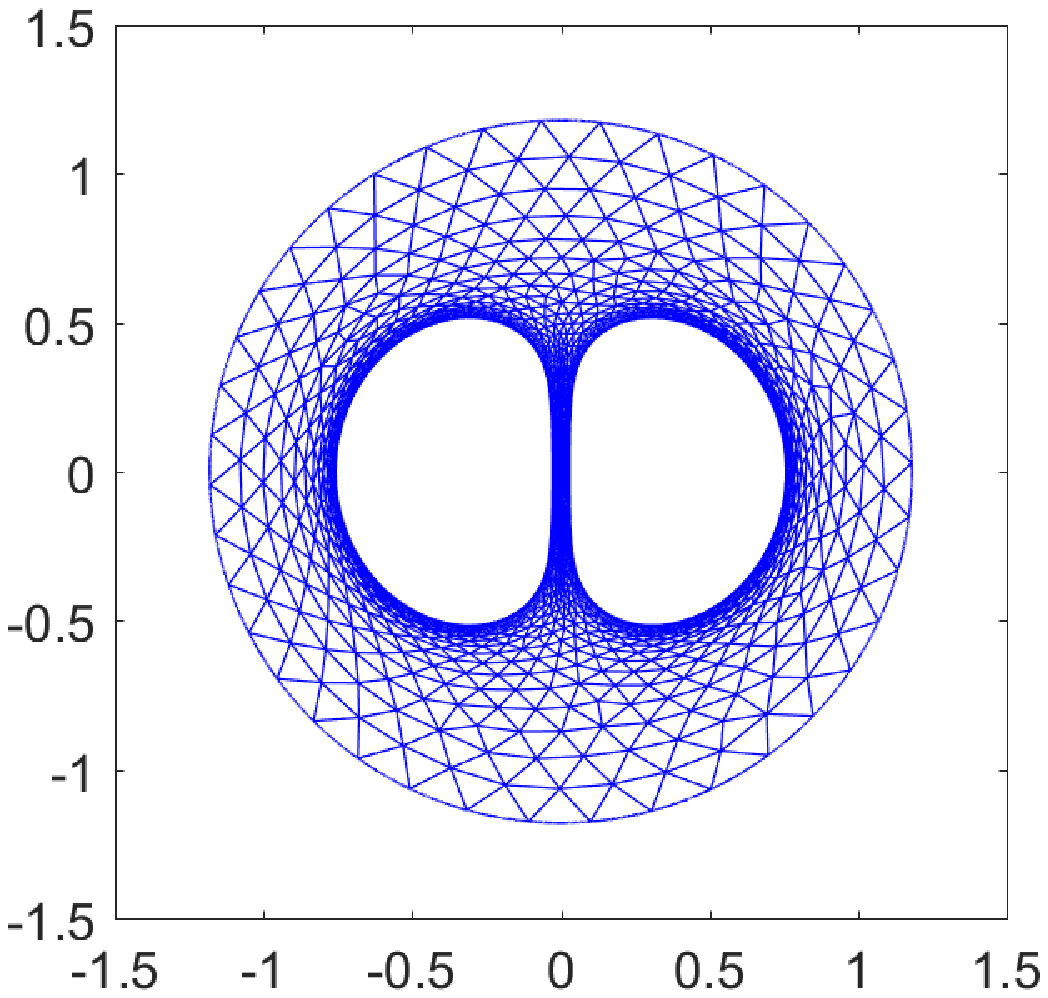}
} \hspace{-12mm}
\centering \subfigure[$\lambda = 1.05$.]{
\label{middle105}
    \includegraphics[width=2.2in, height=1.54in]{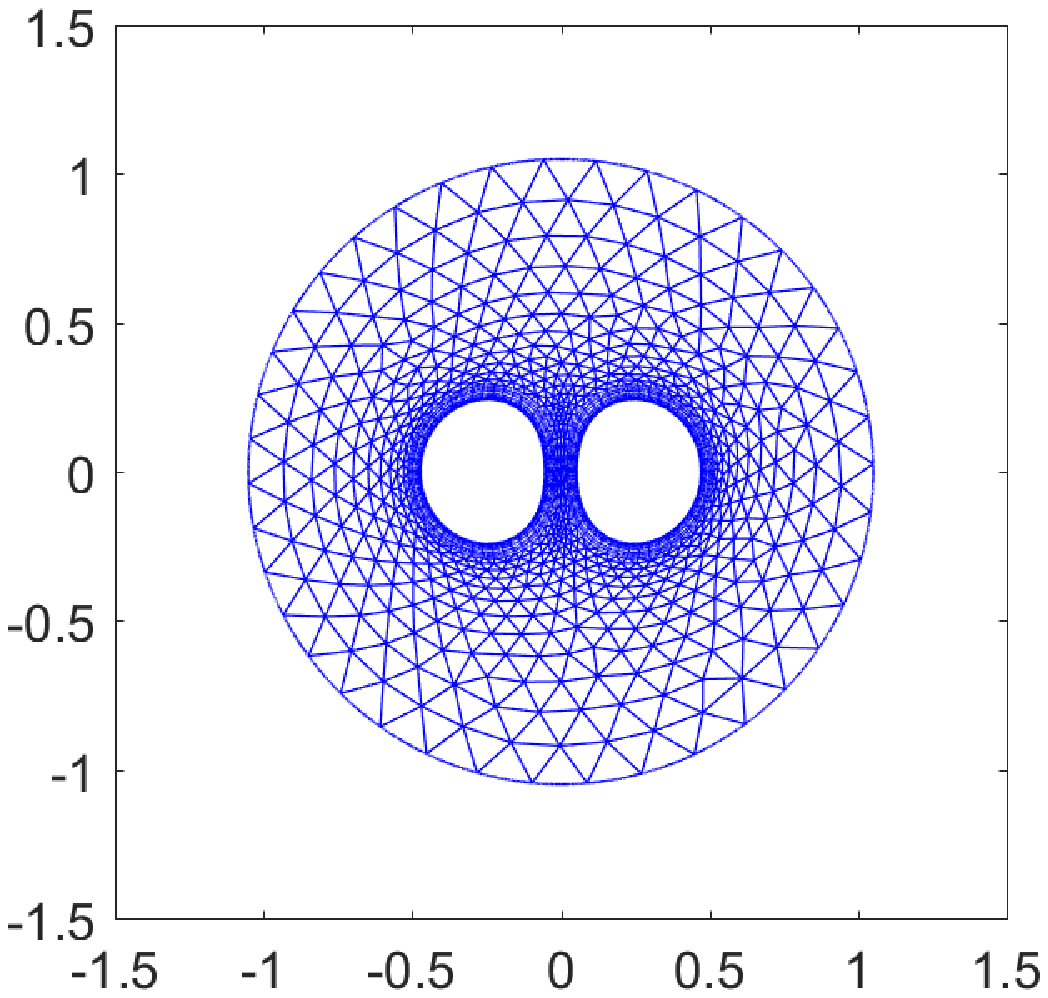}
} \vspace*{-2mm}
  \caption{Axisymmetric cavity solutions, $h = 0.02$.}
  \label{middle evolution}
\end{figure}

\begin{figure}[H]
\centering \subfigure[$\lambda = 1.4$.]{
     \label{right140}
    \includegraphics[width=2.2in, height=1.54in]{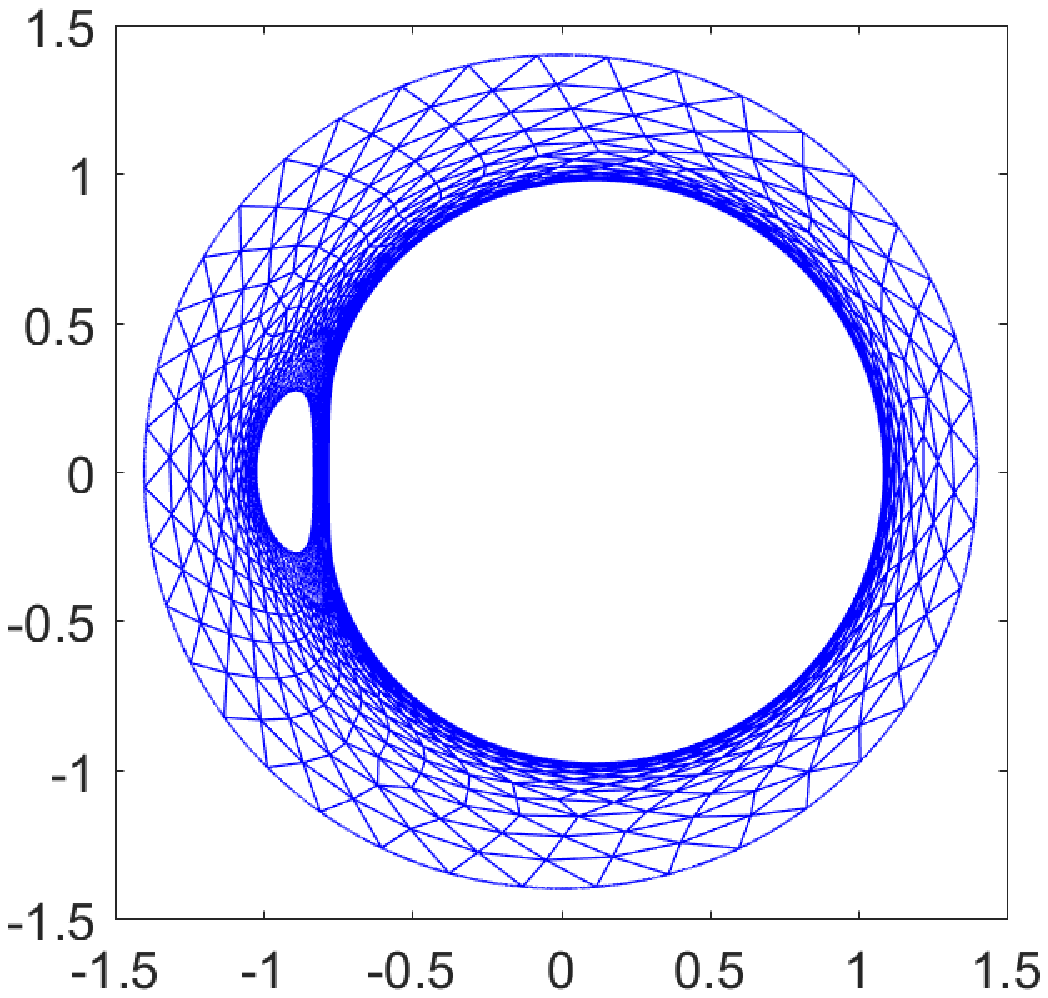}
} \hspace{-12mm}
\centering \subfigure[$\lambda = 1.18$.]{
\label{right118}
    \includegraphics[width=2.2in, height=1.54in]{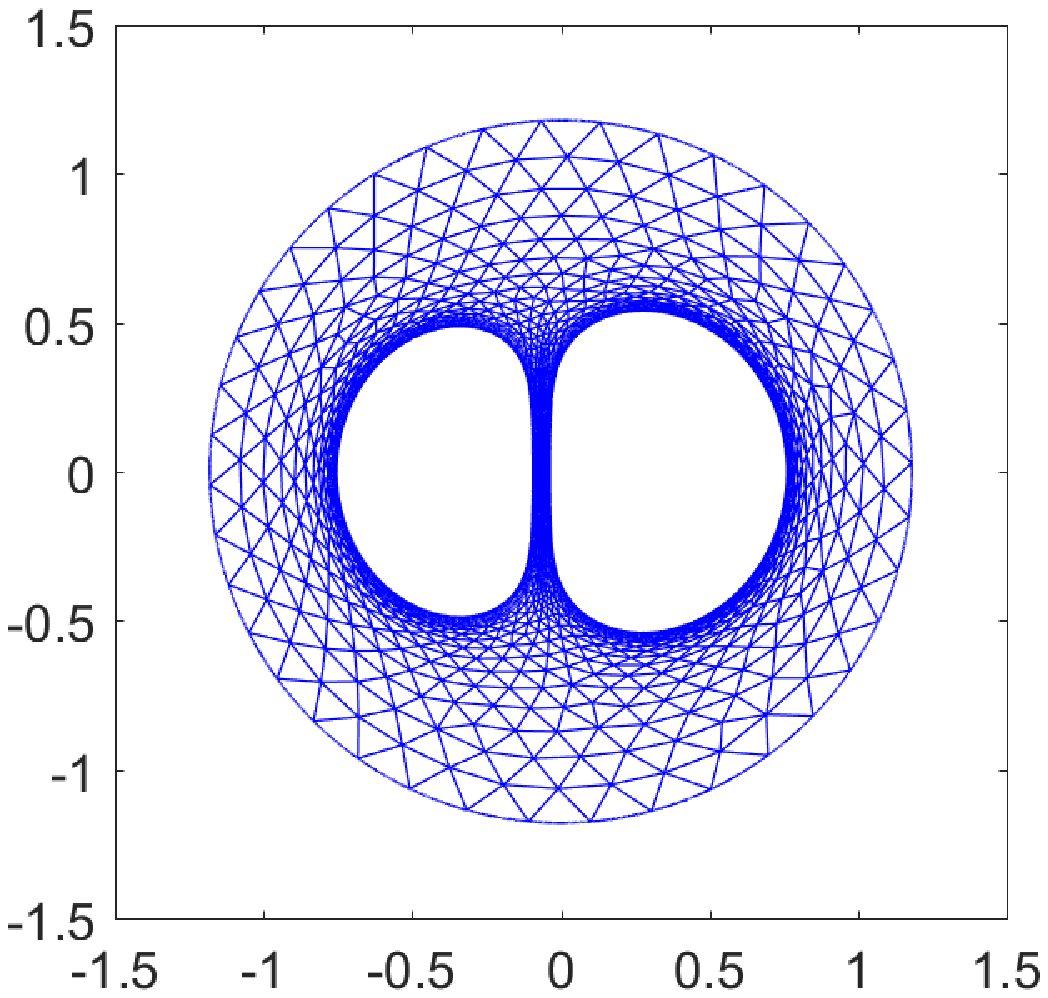}
} \hspace{-12mm}
\centering \subfigure[$\lambda = 1.05$.]{
\label{right105}
    \includegraphics[width=2.2in, height=1.54in]{lambda105.eps}
}
\vspace*{-2mm}
  \caption{The right dominant cavity solutions, $h = 0.02$.}
  \label{right prevails evolution}
\end{figure}

\begin{figure}[H]
%\centering \subfigure[Energy v.s. $h$ ($\lambda=1.3$).]{
%\label{energy convergence}
% \includegraphics[width=2.1in]{energy_convergence_two_defect.eps}
%}\hspace*{-5mm}
\centering \subfigure[Energy difference.]{
\label{energy difference}
 \includegraphics[width=2.5in]{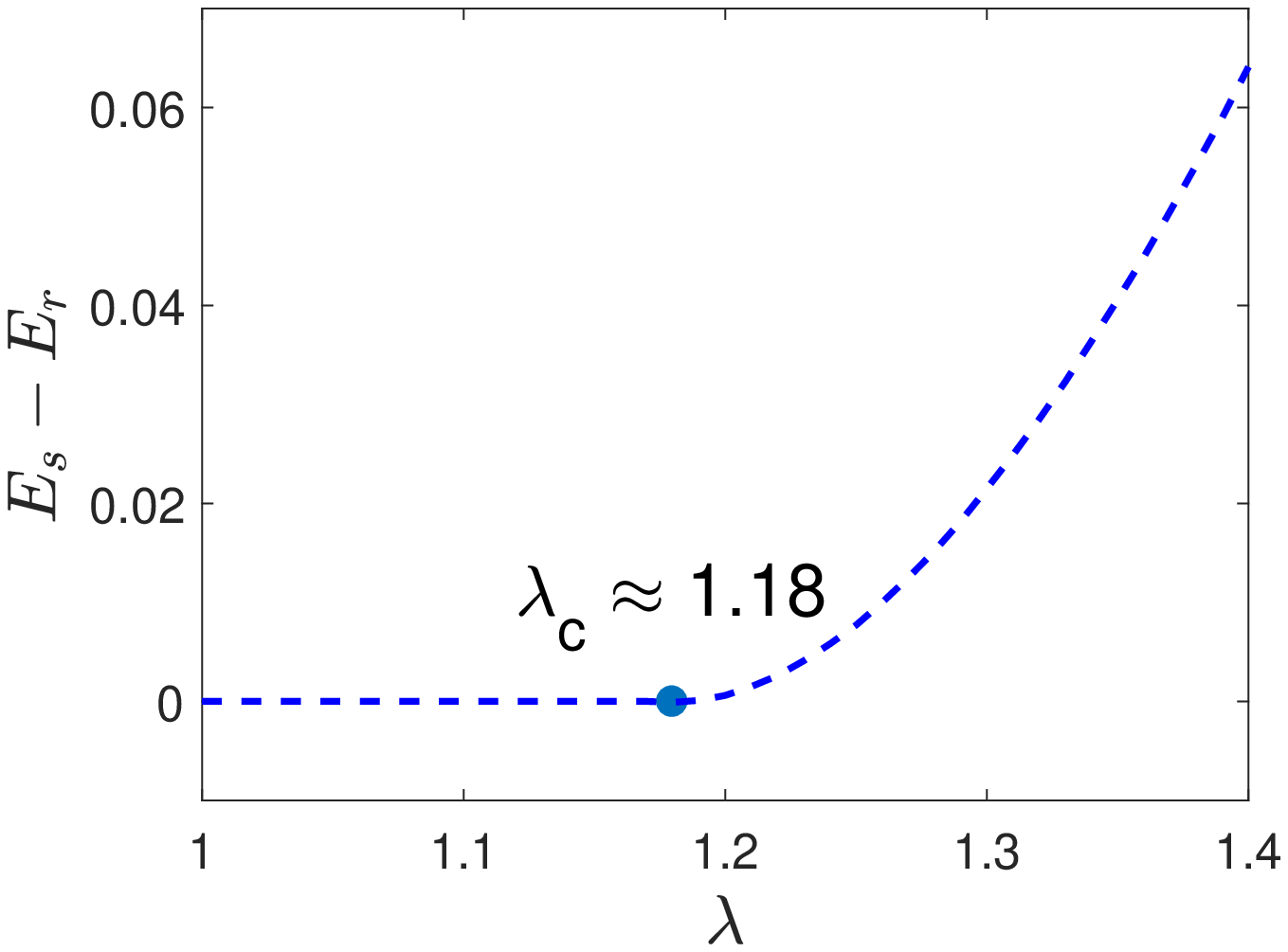}
}\hspace*{8mm}
\centering \subfigure[Volume ratio.]{
\label{volume ratio}
  \includegraphics[width=2.5in]{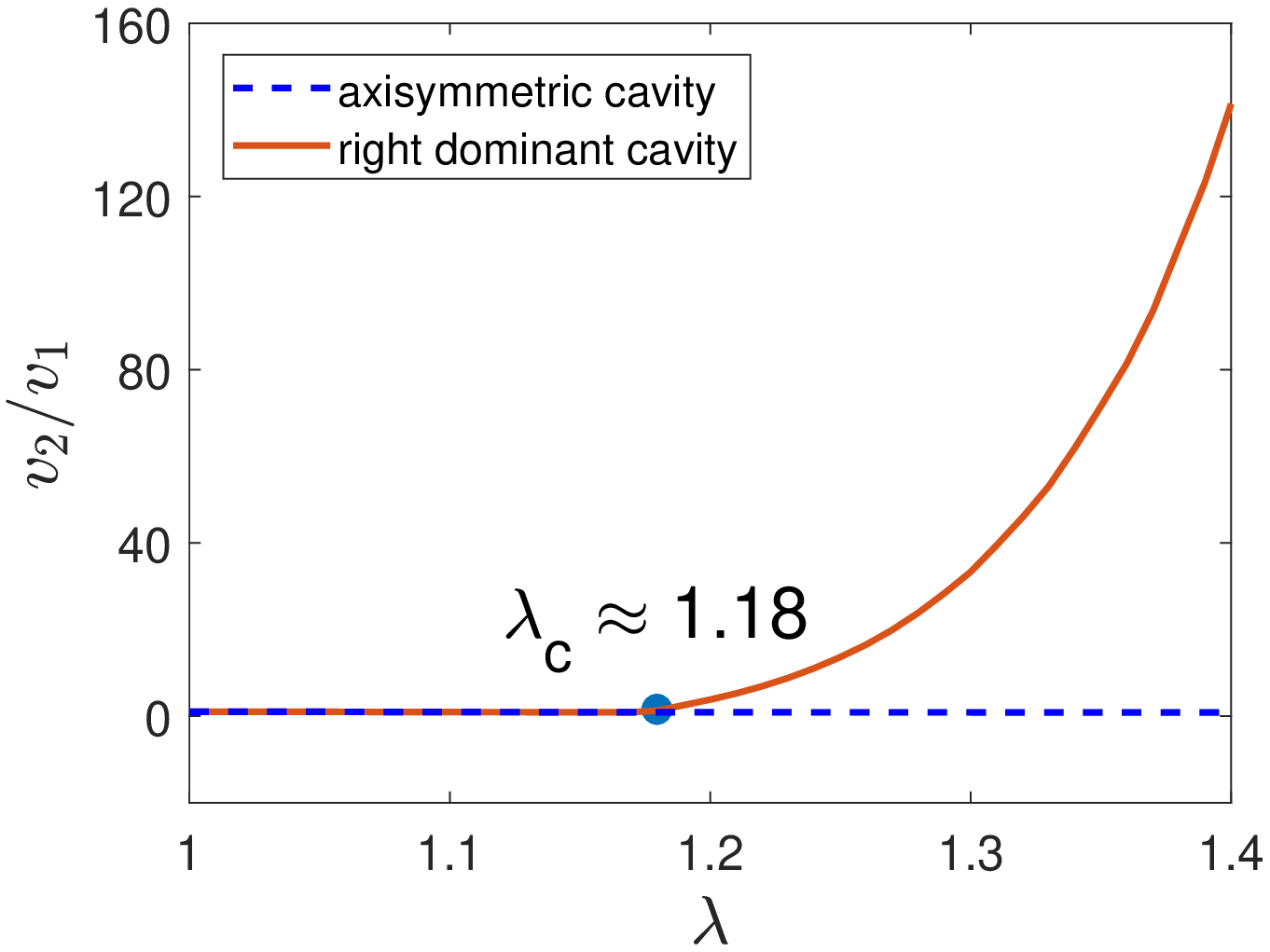}
}
\caption{For $\lambda>\lambda_c$, greater volume ratio cavity solution is energetically favorable.}
   \label{bifurcation}
\end{figure}

Let $E_s$ and $E_r$ denote the elastic energy of the axisymmetric and
right dominant numerical multi-cavity solutions respectively,
Figure~\ref{energy difference} shows $E_s-E_r$ as a function of $\lambda \in [1.0, 1.4]$,
where we see that, for $\lambda <1.18$, $E_s-E_r$ is essentially zero; while
for $\lambda > 1.18$, $E_s-E_r$ grows fast. Let $v_2$ and $v_1$ denote the volumes of the
right and left grown cavities of the multi-cavity solutions respectively. Our numerical
experiments show that $v_2/v_1 \approx 1$ for the axisymmetric numerical cavity solution
for all $\lambda \in [1.0, 1.4]$, and for non-axisymmetric cavity solutions
for all $\lambda \in [1.0, 1.175]$.
Figure~\ref{volume ratio} shows that the volume ratio $v_2/v_1$ of the
right dominant numerical cavity solution grows fast for $\lambda > 1.18$.
Hence we are able to claim that the critical $\lambda_c \in (1.175,1.180)$, and
the greater volume ratio cavity solution is increasingly energetically favorable
for $\lambda>\lambda_c$.

\section{Concluding remarks}

The mixed finite element method introduced in this paper on multi-cavity growth
problem in incompressible nonlinear elasticity is analytically proved to be locking-free
and convergent, and numerically verified to be efficient. In fact, the method enables us
to find a new bifurcation phenomenon in the multi-cavity growth problem.

It is of great interest to further study the bifurcation phenomena,
including those found by Lian and Li in \cite{LianLi2012}, and especially
to explore their relationship with the material failure mechanism.

\bibliographystyle{plain}

\begin{thebibliography}{99}

\bibitem{Gent and Lindley} Gent, A. N., Lindley, P. B., Internal rupture of
bonded rubber cylinders in tension. \textbf{Proc. R. Soc. London, A 249} (1958),
195-205.

\bibitem{Zienkiewicz} Zienkiewicz, O. C., Taylor. R.L., The finite element method
: basic formulation and linear problems, vol. 1 (fourth edition).
\textbf{McGraw-Hill, London} (1989).

\bibitem{SuLi2018} Su, C., Li, Z., A meshing strategy for a quadratic iso-parametric
FEM in cavitation computation in nonlinear elasticity. \textbf{J. Comp. Math. Appl.,
330} (2018), 630-647.

\bibitem{SuLiRectan} Su, C., Li, Z., Error analysis of a dual-parametric bi-quadratic
FEM in cavitation computation in elasticity. \textbf{SIAM J. Numer. Anal., 53(3)}
(2015), 1629-1649.

\bibitem{LianLi2011} Lian, Y., Li, Z., A dual-parametric finite element method
for cavitation in nonlinear elasticity. \textbf{J. Comput. Appl. Math., 236}
(2011), 834-842.

\bibitem{LianLi20112} Lian, Y., Li, Z., A numerical study on cavitations in
nonlinear elasticity-defects and configurational forces.
\textbf{Math. Models Meth. Appl. Sci., 21} (2011), 2551-2574.

\bibitem{LianLi2012} Lian, Y., Li, Z., Position and size effects on voids growth in nonlinear elasticity. \textbf{Int. J. Fract., 173} (2012), 147-161.

\bibitem{HuangLi2017} Huang, W., Li, Z., A Mixed Finite Element Method for Cavitation
Computation in Incompressible Nonlinear Elasticity. \textbf{arXiv:1710.04445v1}.

\bibitem{Fortion1991} Brezzi, F., Fortion, M., Mixed and hybrid finite element methods.
\textbf{Springer-Verlag, New York} (1991).

\bibitem{Wang2013} Shi Z., Wang M., Finite element methods.
\textbf{Science Press, Beijing} (2013).

\bibitem{Henao 2011} Henao, D., Mora, C. C., Fracture surfaces and the regular
of inverses for BV deformations. \textbf{Arch. Rat. Mech Anal., 201} (2011), 575-629.

\bibitem{Henao 2013} Henao, D., Serfaty, S., Energy estimates and cavity interaction for a critical-exponent cavitation model. \textbf{Comm. Pure Appl. Math., 66(7)}, (2013), 1028-1101.

\bibitem{Henao 2010} Henao, D., Mora, C. C., Invertibility and weak continuity
of the determinant for the modelling of cavitation and fracture in nonlinear
elasticity. \textbf{Arch. Rat. Mech. Anal., 197} (2010), 619-655.

\bibitem{Megginson1998} Megginson, R. E., An introduction to Banach space theory.
\textbf{Springer-Verlag, New York} (1998).

\bibitem{Dobrowolski} Dobrowolski, M., A mixed finite element method for approximating incompressible materials. \textbf{SIAM J. Numer. Anal., 29} (1992), 365-389.

\bibitem{Nirenberg1959} Nirenberg, L., On elliptic partial differential equations.
\textbf{Ann. Scuola Norm. Sup. Pisa Sci. Fis. Mat. 13} (1959), 116-162.
\end{thebibliography}

\setlength{\bibsep}{1ex}

\end{document}